\documentclass[10pt]{amsart}

\usepackage{amsmath}
\usepackage{amsthm}
\usepackage{amssymb}
\usepackage[pagebackref=true]{hyperref}

\usepackage[all]{xy}

\usepackage{diagrams}

\theoremstyle{plain}
\newtheorem{thm}{Theorem}[section]
\newtheorem{dfn}[thm]{Definition}
\newtheorem{prp}[thm]{Proposition}
\newtheorem{cor}[thm]{Corollary}
\newtheorem{lma}[thm]{Lemma}
\newtheorem{lma*}[subsection]{Lemma}
\newtheorem{prp*}[subsection]{Proposition}
\newtheorem{dfn*}[subsection]{Definition}

\theoremstyle{remark}
\newtheorem{rmk}[thm]{Remark}
\newtheorem{exms}[thm]{Examples}

\def\Dd{\mathcal{D}}

\def\Ll{\mathcal{L}}

\def\lra{\leftrightarrows}

\def\sg{\sigma}
\def\al{\alpha}
\def\be{\beta}
\def\el{\mathrm{el}}
\def\Aa{\mathcal{A}}
\def\Ii{\mathcal{I}}

\def\inc{\hookrightarrow}
\def\op{\mathrm{op}}

\def\CC{\mathbb{C}}
\def\DD{\mathbb{D}}

\def\Sets{\mathrm{Sets}}

\def\Aut{\mathrm{Aut}}
\def\Oo{\mathcal{O}}

\def\Sg{\Sigma}

\def\LRB{\mathrm{LRB}}
\def\Iso{\mathrm{Iso}}
\def\Oper{\mathrm{Oper}}
\def\Coll{\mathrm{Coll}}
\def\Hom{\mathrm{Hom}}
\def\hgt{\mathrm{ht}}
\def\Xx{\mathfrak{X}}
\def\Yy{\mathfrak{Y}}

\def\nn{\underline{n}}

\def\mm{\underline{m}}
\def\kk{\underline{k}}
\def\11{\underline{1}}
\def\nCat{\mathrm{nCat}}

\def\EE{\mathbb{E}}

\def\el{\mathrm{el}}

\def\FF{\mathbb{F}}

\def\Seg{\mathrm{Seg}}
\def\UU{\mathbb{U}}

\def\Pp{\mathcal{P}}
\def\Sg{\Sigma}

\def\Oo{\mathcal{O}}
\def\FinSet{\mathrm{FinSet}}

\newarrow{Act}--+->
\newarrow{In}{littlevee}--->
\newarrow{convIn}<---<
\newarrow{Inv}{littlevee}----
\newarrow{Identity}=====
\newarrow{DotsIn}{littlevee}...>
\newarrow{DotsAct}..+.>

\newdir{ >}{{}*!/-5pt/@{>}}

\begin{document}

\setcounter{tocdepth}{1}

\title{Moment categories and operads}

\author{Clemens Berger}

\date{December 9, 2022}



\maketitle
\begin{center}\emph{To Bob Rosebrugh, in gratitude and friendship}\end{center}
\begin{abstract}A moment category is endowed with a distinguished set of split idempotents, called moments, which can be transported along morphisms. Equivalently, a moment category is a category with an active/inert factorisation system fulfilling two simple axioms. These axioms imply that the moments of a fixed object form a monoid, actually a left regular band.

Each locally finite unital moment category defines a specific type of operad which records the combinatorics of partitioning moments into elementary ones. In this way the notions of symmetric, non-symmetric and $n$-operad correspond to unital moment structures on $\Gamma$, $\Delta$ and $\Theta_n$ respectively.

There is an analog of the plus construction of Baez-Dolan taking a unital moment category $\CC$ to a unital hypermoment category $\CC^+$. Under this construction, $\CC$-operads get identified with $\CC^+$-monoids, i.e. presheaves on $\CC^+$ satisfying strict Segal conditions. We show that the plus construction of Segal's category $\Gamma$ embeds into the dendroidal category $\Omega$ of Moerdijk-Weiss.\end{abstract}

\tableofcontents

\section*{Introduction}

What is an operad ? Although the pioneering work of May \cite{May} and Boardman-Vogt \cite{BV} is half a century old, the question is more intricate than it might seem at first sight. A multitude of types of operads have appeared (symmetric, non-symmetric, cyclic, modular, coloured, ...) and are used in different areas of mathematics and even outside. A common feature is the existence of a process of substitution. Here we follow an unconventional route, not based on the concept of substitution, and not intended to enlarge further the panorama of operadic structures, but rather to look at them from a different perspective.

Starting point is the existence of an \emph{active/inert factorisation system} with suitable properties. The basic example is Segal's category $\Gamma$, the dual of the category of finite sets and partial maps. Our terminology has been motivated by Lurie \cite{Lu} who uses extensively the inert/active factorisation system of $\Gamma^\op$. Active morphisms in $\Gamma$ can be viewed as partitions of the target, indexed by the elements of the source. Inert morphisms are simply inclusions. Segal's motivation \cite{Se} to choose $\Gamma$ comes from the existence of a canonical covariant functor $\gamma_\Delta:\Delta\to\Gamma$ linking simplicial combinatorics to $\Gamma$. There is an active/inert factorisation system for $\Delta$ compatible with this functor: active morphisms are endpoint-preserving, inert morphisms distance-preserving. By means of a wreath product \cite{Be2} this active/inert factorisation system carries over to Joyal's categories $\Theta_n$ \cite{J}.

In all three examples, inert morphisms have \emph{unique} active retractions. This produces split idempotent endomorphisms, called \emph{moments}, in bijective correspondence with inert subobjects. We call the whole structure a \emph{moment category}. Moreover, in all three cases, there is a well-defined object with a single centric moment, called \emph{unit}. Inert subobjects are called \emph{elementary} if they have a unit as domain. It turns out that the scheme according to which moments decompose into elementary moments, defines an operad-like structure, and this is so for any unital moment category of finite type. $\Delta$-operads are non-symmetric operads, $\Gamma$-operads are symmetric operads, and $\Theta_n$-operads are Batanin's $n$-operads \cite{Ba}.

There is an essentially unique augmentation $\gamma_\CC:\CC\to\Gamma$ for each unital moment category $\CC$ of finite type. This functor can be considered as a notion of \emph{cardinality}. It also suggests that active morphisms in $\CC$ are generalised partitions, and inert morphisms generalised inclusions. Taking the existence of such an augmentation as basic leads to the more flexible notion of \emph{hypermoment category} where the notion of moment is superseded by the notion of inert subobject.

There are three interesting examples of hypermoment categories in literature. The dendroidal category $\Omega$ of Moerdijk-Weiss \cite{MW}, the graphoidal category $\Gamma_\updownarrow$ of Hackney-Robertson-Yau \cite{HRY} as well as the undirected analog $\UU$ recently introduced by the same authors \cite {HRY2,H}. In all three cases, the inert morphisms are inclusions, often referred to as \emph{outer} face operators. The active morphisms are either degeneracy operators or \emph{inner} face operators. The latter can be understood geometrically as insertion of dendrices (resp. graphices, resp. graphs) into vertices of dendrices (resp. graphices, resp. graphs). The augmentation takes an object to its vertex set. $\Omega$-operads (resp. $\Gamma_\updownarrow$-operads, resp. $\UU$-operads) are tree-hyperoperads (resp. directed hyperoperads, resp. hyperoperads) in a sense close to the original notion of Getzler-Kapranov \cite{GK} bearing the same name. This motivated our terminology.

The terminal $\CC$-operad plays a special role, and we call algebras over the terminal $\CC$-operad \emph{$\CC$-monoids}. The structure of $\CC$-monoid is interesting in its own because $\CC$-monoids are ``special'' presheaves on the active part of $\CC$ and are easier to describe than operads. Thanks to the inert part of $\CC$, the notion of $\CC$-monoid can be reformulated by means of strict Segal conditions. It is then natural to define $\CC_\infty$-monoids as simplicial presheaves on $\CC$ subject to homotopical Segal conditions. The table below (copied from \ref{table}) summarises the notions we get in this way.

\begin{center}\begin{tabular}{|c|c|c|c|c|}
\hline $\CC$ & $\CC$-operad & $\CC$-monoid & $\CC_\infty$-monoid & group-like $\CC_\infty$-monoid\\
\hline $\Gamma$ & sym. operad & comm. monoid & $E_\infty$-space & infinite loop space\\
\hline $\Delta$& non-sym. operad & assoc. monoid & $A_\infty$-space & loop space\\
\hline $\Theta_n$&$n$-operad & $n$-monoid & $E_n$-space&$n$-fold loop space\\
\hline $\Omega$ & tree-hyperoperad & sym. operad & $\infty$-operad& (stable $\infty$-operad)\\
\hline $\Gamma_{\updownarrow}$ & directed hyperoperad & properad & $\infty$-properad& (stable $\infty$-properad)\\
\hline $\UU$ & hyperoperad & modular operad & $\infty$-modular op. & (stable $\infty$-modular op.)\\
\hline\end{tabular}\end{center}\vspace{1ex}

Most notably, symmetric operads appear twice, as $\Gamma$-operads and as $\Omega$-monoids. This reveals a tight relationship between $\Gamma$ and $\Omega$ implicitly present at several places in literature \cite{HHM,Ba,CHH}. We deduce this relationship from an analog of the \emph{plus construction} of Baez-Dolan \cite{BD}. For hypermoment categories $\CC$, the plus construction $\CC^+$ is defined as a category of \emph{abstract $\CC$-trees}, which are composable chains of active morphisms of $\CC$ starting with a unit of $\CC$. The inert part of $\CC$ contributes to the inert part of $\CC^+$. For $\CC=\Gamma$ it turns out that the plus construction $\Gamma^+$ embeds as a non-full hypermoment subcategory into $\Omega$. While $\Gamma$ has a single unit, $\Gamma^+$ and $\Omega$ have a unit for each natural number, namely the corolla with $n$ leaves. It is a pleasant feature that the units of a hypermoment category are always determined intrinsically by the active/inert factorisation system, so that there is no choice here. The nilobjects (i.e. objects of cardinality $0$) as well are intrinsically given. The plus construction converts units of $\CC$ into nilobjects of $\CC^+$ and general objects of $\CC$ into units of $\CC^+$. This reproduces the general scheme suggested by Baez-Dolan \cite{BD}. Under the plus construction, $\CC$-operads get identified with $\CC^+$-monoids. Therefore, homotopy $\CC$-operads can be modeled as $(\CC^+)_\infty$-monoids, i.e. as simplicial presheaves on $\CC^+$ satisfying homotopical Segal conditions. An incarnation of this idea is the Cisinski-Moerdijk model structure \cite{CM} on $\Omega$-spaces. Indeed, the combinatorial difference between $\Gamma^+$ and $\Omega$ disappears at a homotopical level as follows from \cite[Theorem 5.1]{CHH} of Chu-Haugseng-Heuts.

Active and inert parts of a hypermoment category $\CC$ interact via the factorisation system: the active part underlies the substitutional aspect of $\CC$-operads while the inert part is responsible for the homotopical aspects of $\CC$-operads. We discuss two structural properties of a hypermoment category: strong unitality and extensionality. \emph{Strong unitality} means that the unit- and nilobjects of the hypermoment category are dense in the inert part. This defines density colimit cocones, and a simplicial presheaf on $\CC$ defines a $\CC_\infty$-monoid (i.e. a \emph{Segal presheaf} on $\CC$) if and only if it takes these density colimit cocones to homotopy limit cones.

This abstract condition amounts precisely to Segal's notion \cite{Se} of ``special'' $\Delta$-, resp. $\Gamma$-space. The \emph{same} condition is the key in Rezk's model structure \cite{R} for $\Theta_n$-spaces, and also in Cisinski-Moerdijk's model structure \cite{CM} for $\Omega$-spaces. This common thread suggests that for any strongly unital hypermoment category $\CC$ an analogous model structure exists whose fibrant objects are $\CC_\infty$-monoids. Chu-Haugseng \cite{CH} develop a closely related concept in the $\infty$-categorical framework: \emph{extendable algebraic patterns}. Although their axiomatics are somewhat different, it is striking that their examples are identical to ours.

A unital hypermoment category is called \emph{extensional} if elementary inert morphisms admit pushouts along active morphisms in a compatible way with the active/inert factorisation system. Again, this is a property shared by all examples so far mentionned. Extensionality implies that there is a well-defined process of inserting $\CC$-trees into vertices of $\CC$-trees. This is one of the essential ingredients needed for the forgetful functor from $\CC$-operads to $\CC$-collections to be \emph{monadic}, despite of the complicated structure the symmetries of a $\CC$-collection may have. On the other hand, if the amount of ``allowed'' symmetries is restricted (in which case the hypermoment category is said to be \emph{rigid}) in such a way that the active/inert factorisation becomes unique when the domain is a unit, then the dual of the active part carries the structure of an \emph{operadic category} in the sense of Batanin-Markl \cite{BaM}. We get valuable examples of operadic categories in this way. Note that the moment categories $\Delta,\Gamma$ and $\Theta_n$ are rigid, while the hypermoment categories $\Omega,\Gamma_\updownarrow$ and $\UU$ are not.

We have limited ourselves to the combinatorial aspects of moment categories and hope to pursue homotopical applications elsewhere. In Barwick's article \cite{Ba} much abstract homotopy theory is developed in the setting of his operator categories. The algebraic patterns of Chu-Haugseng \cite{CH} are even closer to our approach.\vspace{1ex}

Let us now describe the contents of the individual sections:\vspace{1ex}

Section 1 defines moment categories. Centric moment categories are studied in some detail because they relate to restriction categories in the sense of Cockett-Lack \cite{CL1}. We present an alternative characterisation of moment categories in terms of the existence of pushforward operations transporting moments across morphisms. Our axioms are less restrictive than those of Cockett-Lack insofar as moments are not required to commute with each other. It is remarkable that the axioms entail nevertheless that moments of a fixed object form a so called left regular band.

Section 2 deals with unital moment categories and defines their operads and monoids. Moments splitting over a unit are called elementary, and pushforward takes ``disjoint'' moments to ``disjoint'' moments. This implies that each active morphism induces a partition of the identity moment of the target into submoments indexed by the elementary moments of the source. Composition of active morphisms induces then an operad-like structure. Every unital moment category of finite type is shown to be augmented over Segal's category $\Gamma$.

Section 3 studies hypermoment categories with special emphasis on the dendroidal category $\Omega$ and the graphoidal category $\Gamma_\updownarrow$. The plus construction is introduced showing that it transforms $\CC$-operads into $\CC^+$-monoids. We then discuss strongly unital and extensional hypermoment categories. For extensional hypermoment categories $\CC$, the objects of the plus construction $\CC^+$, the so-called $\CC$-trees, can be inserted into vertices of $\CC$-trees. This ultimately implies that the forgetful functor from $\CC$-operads (or, equivalently, $\CC^+$-monoids) to $\CC$-collections is monadic.

Appendix \ref{operadicappendix} shows that the dual of the active part of a rigid hypermoment category is an operadic category in a canonical way.

Appendix \ref{dendrixappendix} contains a combinatorial description of the embedding of $\Gamma^+$ into $\Omega$, identifying the category of $\Gamma$-trees with the category of reduced dendrices.

\subsection*{Acknowledgements}This text would not exist without the many instructive and enlightening discussions I had over the years with Clark Barwick, Michael Batanin, Richard Garner, Ralph Kaufmann, Steve Lack, Ieke Moerdijk and Mark Weber.

I'm also grateful to the referee for his careful reading of the manuscript which allowed me to eliminate several inaccuracies and to improve considerably the presentation of the section on hypermoment categories. This research benefitted from financial support by the ERC-project DuaLL of Mai Gehrke.

\section{Moment categories}

This introductory section contains slightly more than absolutely necessary for the application to operads in the following sections. From a purely abstract point of view, moment categories can be considered as categorification of left regular bands, well known in semigroup literature, cf. \cite{MSS}. It is remarkable that the construction of the universal commutative quotient of a left regular band extends to moment categories. Because the underlying notion of commutativity is quite subtle we introduce the term \emph{centric} for it. Centric moment categories are dual to split restriction categories in the sense Cockett-Lack \cite{CL1}.

\subsection{Active/inert factorisation systems}A class of morphisms in a category $\CC$ is said to be \emph{closed} if it is closed under composition and contains all isomorphisms of $\CC$. A subcategory of $\CC$ is said to be \emph{wide} if it contains all objects and all isomorphisms of $\CC$. Any closed class of morphisms in $\CC$ defines a wide subcategory of $\CC$, and conversely. For ease of exposition we will tacitly identify closed classes with the corresponding wide subcategories.

A $(\CC_{act},\CC_{in})$-\emph{factorisation system} on a category $\CC$ consists of two \emph{closed} classes $\CC_{act}$ and $\CC_{in}$ such that every morphism in $\CC$ factors in an \emph{essentially unique} way as $f=f_{in}f_{act}$ where $f_{act}$ belongs to $\CC_{act}$ and $f_{in}$ belongs to $\CC_{in}$, i.e. each morphism may be written essentially uniquely as a morphism in $\CC_{act}$ followed by a morphism in $\CC_{in}$. Essential uniqueness means that for any two factorisations $f_{in}f_{act}=f'_{in}f'_{act}$ there is a \emph{unique} isomorphism $h$ such that $f_{in}=f'_{in}h$ and $hf_{act}=f'_{act}$. Uniqueness of $h$ is automatic if either $\CC_{act}$ consists of epimorphisms or $\CC_{in}$ consists of monomorphisms. It is equivalent to the condition that the morphisms in $\CC_{act}$ are left orthogonal to the morphisms in $\CC_{in}$.

Throughout this text, the morphisms in $\CC_{act}$ will be called \emph{active}, the morphisms in $\CC_{in}$ \emph{inert}, and the factorisation system itself will be called an \emph{active/inert factorisation system}. The wide subcategory associated to $\CC_{act}$ (resp. $\CC_{in}$) is the \emph{active} (resp. \emph{inert}) \emph{part} of $\CC$. The axioms of a moment category will imply that inert morphisms are split monomorphisms, that epimorphisms are active, but that in general active morphisms do not need to be epimorphic. In Section \ref{hyper} we will introduce \emph{hypermoment categories}. These are equipped with an active/inert factorisation where inert morphisms are not necessarily split monomorphisms.

\begin{dfn}A \emph{moment category} is a category equipped with an active/inert factorisation system such that

\begin{itemize}\item[(M1)]every inert morphism has a unique active retraction;\item[(M2)]if $fi=g$ for an inert morphism $i$ and active morphisms $f,g$ then $f=gr$ where $r$ is the unique active retraction of $\,i$ provided by \emph{(M1)}.\end{itemize}

\noindent A moment category is called \emph{centric} if \begin{itemize}\item[(MC)]every active morphism has at most one inert section.\end{itemize}

A \emph{moment} of an object $A$ is an endomorphism $\phi$ of $A$ such that if $\phi=\phi_{in}\phi_{act}$ then $\phi_{act}\phi_{in}=1_B$ for an object $B$. We shall say that the moment $\phi\,$ \emph{splits over} $B$.\vspace{1ex}

A morphism is called \emph{retractive} if it is active and admits an inert section.\vspace{1ex}

A moment functor is a functor preserving active and inert morphisms.\end{dfn}

\noindent It follows from (M1) and the essential uniqueness of active/inert factorisations that any morphism which is active \emph{and} inert must be an isomorphism, hence the intersection $\CC_{act}\cap\CC_{in}$ is the closed class $\CC_{iso}$ of isomorphisms of $\CC$.\vspace{1ex}

We shall in general denote active morphisms by arrows of the form $\xymatrix{{}\ar[r]|+&}$ and inert morphisms by arrows of the form $\xymatrix{{}\ar@{ >->}[r]&}$. Axiom (M2) is equivalent to the following  axiom (M2)$'$ which is mnemotechnically easier to retain: \begin{itemize}\item[(M2)$'$]\emph{If the left square below commutes then the right square as well $$\xymatrix{A\ar[r]^f|+&B&&A\ar[d]_r|+\ar[r]^f|+&B\ar[d]^{r'}|+\\
A'\ar@{ >->}[u]^{i}\ar[r]_{g}|+&B'\ar@{ >->}[u]_{i'}&&A'\ar[r]_{g}|+&B'}$$where $r,r'$ are the unique active retractions of $i,i'$ provided by (M1).}\end{itemize}\vspace{1ex}

Each moment of $A$ satisfies $\,\phi\phi=\phi_{in}\phi_{act}\phi_{in}\phi_{act}=\phi_{in}\phi_{act}=\phi$ and is thus a \emph{split idempotent endomorphism} of $A$. We shall call an isomorphism class of inert morphisms with fixed target $A$ an \emph{inert subobject} of $A$.

\begin{lma}\label{inertsubobject}For each object of a moment category there is a canonical bijection between its moments and its inert subobjects.\end{lma}
\begin{proof}Two splittings of an idempotent endomorphism take place over canonically isomorphic objects so that each moment of $A$ defines an inert subobject of $A$. Conversely, by (M1), each inert morphism $i:B\to A$ generates a moment $ir:A\to A$ from which it derives, and isomorphic inert morphisms generate the same moment.\end{proof}

\subsection{Pushing forward moments}\label{pushforward}Moments have the advantage over inert subobjects that there is no need to quotient by any equivalence relation. The moments of an object form a subset of its endomorphism set.

It is important to observe that while moments are in bijection with inert subobjects it is in general not true that moments are also in bijection with retractive quotients, since a retractive morphism may have several distinct inert sections. Such a bijection holds if the moment category is \emph{centric}, i.e. fulfills axiom (MC).

It turns out that \emph{centric moment categories} have already been studied in literature since they correspond bijectively to \emph{split corestriction categories} in the sense of Cockett-Lack \cite{CL1}. We shall make this correspondence explicit, since it helps us to introduce useful terminology and notation, and gives us an alternative definition of a moment category in terms of ... its moments !

Let $f:A\to B$ be a morphism in a moment category and $\phi$ a moment of $A$. Choose a splitting $\phi=\phi_{in}\phi_{act}$ and denote the object over which $\phi$ splits by $A'$. Then factor the composite morphism $f\phi_{in}:A'\to B$ into an active morphism $f':A'\to B'$ followed by an inert morphism $\psi_{in}:B'\to B$. We denote the unique active retraction of $\psi_{in}$ by $\psi_{act}:B\to B'$ and the associated moment by $\psi=\psi_{in}\psi_{act}$.

The \emph{pushforward of the moment $\phi$ along $f$} is then defined by $f_*(\phi)=\psi$. The following diagram summarises the construction:\begin{gather}\label{diagram1}\begin{diagram}[small,silent]A&\rTo^f&B&&\\\dAct^{\phi_{act}}\uIn_{\phi_{in}}&&\uIn^{\psi_{in}}\dAct_{\psi_{act}}&\quad&\quad \text{with}\quad f_*(\phi_{in}\phi_{act})=\psi_{in}\psi_{act}.\\A'&\rAct_{f'}&B'&&\end{diagram}\end{gather}
 Observe that the isomorphism type of the morphism $\psi_{in}:B'\to B$ (with fixed $B$) only depends on the isomorphism type of the morphism $\phi_{in}:A'\to A$ (with fixed $A$), i.e. $\psi$ is uniquely determined by the morphism $f$ and the moment $\phi$.

By construction, the inner square is commutative. In order to show that the outer square is commutative as well, we decompose (\ref{diagram1}) as follows:
\begin{gather}\label{diagram2}\begin{diagram}[small]A&\rAct^{f_{act}}&B''&\rIn^{f_{in}}&B\\\dAct^{\phi_{act}}\uIn_{\phi_{in}}&&
\uIn^{\xi_{in}}\dAct_{\xi_{act}}&&\uIn^{\psi_{in}}\dAct_{\psi_{act}}\\A'&\rAct_{f'}&B'&\rIdentity&B'\end{diagram}\end{gather}
Observe that the active/inert factorisation of $f\phi_{in}$ can be obtained by composing the active/inert factorisation of $f_{act}\phi_{in}$ with $f_{in}$, hence we can assume $f_{in}\xi_{in}=\psi_{in}$. Therefore, the unique active retraction $\xi_{act}$ of $\xi_{in}$ is the composite of the unique active retractions of $f_{in}$ and $\xi_{in}$. In particular, $\psi_{act}f_{in}=\xi_{act}$ and the right hand square is commutative. It suffices now to show that $\xi_{act}f_{act}=f'\phi_{act}$. This follows from axiom (M2)$'$ since the left inner square of (\ref{diagram2}) commutes by construction.

As corollary we obtain for each moment $\phi$ of $A$ and each morphism $f:A\to B$ the important relation $f\phi=f_*(\phi)f$. Moreover, the essential uniqueness of active/inert factorisations implies that pushforward is functorial in the following sense: for $f:A\to B$ and $g:B\to C$ we have $(gf)_*(\phi)=g_*(f_*(\phi))$. Finally, it follows from the definition of the pushforward that $\phi_*(\psi)=\phi\psi$ for any two moments of the same object. This leads to the following definition:

\begin{dfn}\label{formaldfn}A \emph{moment structure} on a category consists in specifying for each object $A$, a set $m_A$ of special endomorphisms of $A$, called \emph{moments}, and for each morphism $f:A\to B$, a \emph{pushforward operation} $f_*:m_A\to m_B$ such that the following four axioms hold (for any $A$, any $\phi,\psi\in m_A$ and $f:A\to B,\,g:B\to C)$:
\begin{itemize}\item[(m1)] $1_A\in m_A$\item[(m2)]$\phi_*(\psi)=\phi\psi$ \item[(m3)]$(gf)_*=g_*f_*$\item[(m4)]$f\phi=f_*(\phi)f$\end{itemize}

A morphism $f:A\to B$ is called \emph{active} (resp. \emph{inert}) if $f_*(1_A)=1_B$ (resp. if $f$ admits a retraction $r:B\to A$ such that $f_*(\phi)=f\phi r$ for all $\phi\in m_A$).

A moment of $A$ is said to \emph{split over $B$} if there exists $i:B\lra A:r$ such that $ir=\phi$ and $ri=1_B$.\end{dfn}

\begin{lma}\label{formal}A category with moment structure enjoys the following properties:
\begin{itemize}\item[(i)]The moment set $\,m_A$ is a \emph{monoid} under composition such that $\phi\psi=\phi\psi\phi$ for all $\phi,\psi\in m_A$. In particular, each moment is \emph{idempotent}.\item[(ii)]For any $\phi,\psi\in m_A$ the relations $\psi\phi=\phi$ and  $\phi\psi=\psi$ jointly imply $\phi=\psi$.\item[(iii)]An endomorphism $\phi$ of $A$ belongs to $m_A$ if and only if $\phi_*(1_A)=\phi$.\\ For any morphism $f:A\to B$ and any $\phi\in m_A$ one has $f_*(\phi)=(f\phi)_*(1_A)$.\item[(iv)]Epimorphisms are active.\item[(v)]For any splitting $i:B\lra A:r$ of a moment $\phi\in m_A$, the retraction $r$ is active and the section $i$ is inert. In particular, $r_*(\phi)=1_B$ and  $i_*(1_B)=\phi$.\item[(vi)]The class of active (resp. inert) morphisms is closed.\item[(vii)]$f$ has an active/inert factorisation if and only if the moment $f_*(1_A)$ splits.\end{itemize}\end{lma}

\begin{proof}--

(i) By (m1) and (m2), the moment set $m_A$ is a monoid. (m2) and (m4) imply the relation $\phi\psi\phi=\phi\psi$. Putting $\psi=1_A$ yields $\phi=\phi^2$.

(ii) By (i), $\phi=\psi\phi=\psi\phi\psi=\phi\psi=\psi.$

(iii) By (m2), any moment satisfies $\phi=\phi 1_A=\phi_*(1_A)$. It follows then from (m3) that $f_*(\phi)=f_*(\phi_*(1_A))=(f\phi)_*(1_A)$.

(iv) By (m4), one has $1_Bf=f 1_A=f_*(1_A)f$. Thus if $f$ is epic then $f_*(1_A)=1_B$.

(v) Since $r$ is active by (iv), we get by (iii)$$r_*(\phi)=(r\phi)_*(1_A)=(rir)_*(1_A)=r_*(1_A)=1_{B}.$$For each $\psi\in m_B$, (m4) implies $i_*(\psi)ir=i\psi r$. Since $m_A$ is a monoid by (i), and $ir=\phi$ it follows that $i\psi r\in m_A$. Therefore, by (iii) and (m2), $i\psi r=(i\psi r)_*(1_A)=i_*(\psi)$ so that $i$ is inert. In particular, $i_*(1_B)=ir=\phi$.

(vi) Both classes are closed under composition. Isomorphisms are active by (iv) and inert by (m4) since their pushforward action is the conjugation action.

(vii) Assume first that $f=ig$ with $g$ active and $i$ inert and denote by $r$ a retraction of $i$ such that $ir$ is a moment of $A$. Then we get by (m3) and the definition of active (resp. inert) morphisms that $f_*(1_A)=i_*(g_*(1_A))=ir$. Conversely, if the moment $f_*(1_A)$ splits as $f_*(1_A)=ir$, put $g=rf$. Since by (m4) $f=f1_A=f_*(1_A)f$ we get $f=ig$ where $i$ is inert by (v). Moreover, by (m3) and (v), we get $g_*(1_A)=r_*(f_*(1_A))=r_*(ir)=1$ hence $g$ is active.\end{proof}

\begin{prp}\label{formalsplit}A moment category is the same as a category with moment structure in which all moments split.\end{prp}

\begin{proof}We have seen that each moment category induces a moment structure by specifying as moments those endomorphisms for which the active part is a retraction of the inert part. Indeed, the factorisation system defines a pushforward operation by diagram (\ref{diagram1}) above, which satisfies the axioms (m1), (m2), (m3), (m4). Note that active (resp. inert) morphisms of the factorisation system are indeed active (resp. inert) in the sense of Definition \ref{formaldfn}, and that by definition all moments split.

Conversely, given a category with moment structure in which all moments split, the active/inert factorisation system derives from Lemma \ref{formal}vi and vii. The active/inert factorisation is essentially unique because inert morphisms have retractions.

It remains to be shown that the axioms (M1) and (M2) of a moment category hold. Assume that an inert morphism $i:B\to A$ has active retractions $r,s:A\to B$ with moments $\phi=ir$ and $\psi=is$. These moments of $A$ are mutually ``right-absorbing'', i.e. $\phi\psi=\psi$ and $\psi\phi=\phi$. Therefore, by Lemma \ref{formal}ii, $\phi=\psi$ and hence $r=s$. Assume finally that $fi=g$ for an inert morphism $i$ and active morphisms $f,g$. Then for the (unique) active retraction $r$ of $i$ we get $gr=fir=f_*(ir)f=f$ where the last equality follows from the hypothesis that $g$ and hence $gr=fir$ are active so that by Lemma \ref{formal}iv, $f_*(ir)=(fir)_*(1)=1$.\end{proof}

\begin{dfn}\label{geometric}Let $\phi,\psi$ be moments of the same object. The moment $\phi$ is said to be a \emph{submoment} of $\psi$ if $\psi\phi=\phi$ in which case we write $\phi\leq\psi$.

Two moments $\phi,\psi$ are said to be \emph{congruent} if $\phi=\phi\psi\phi$ and $\psi=\psi\phi\psi$, in which case we write $\phi\simeq\psi$.

A moment is called \emph{centric} if its congruence class is singleton. A moment structure is called \emph{centric} if all its moments are centric.\end{dfn}

The relation $\leq$ on $m_A$ is reflexive and transitive. By Lemma \ref{formal}ii it is also antisymmetric and defines thus a partial order relation on the local monoid $m_A$.

\begin{lma}\label{Greenorder}Let $A$ be an object of a category with moment structure.

\begin{itemize}\item[(i)]If all moments of $A$ are split, the poset $(m_A,\leq)$ is isomorphic to the poset of inert subobjects of $A$ ordered by inclusion;\item[(ii)]For any $f:A\to B$ and $\phi,\psi\in m_A$ we have $f_*(\phi\psi)=f_*(\phi)f_*(\psi)$. In particular, pushforward $f_*:m_A\to m_B$ is order-preserving;\item[(iii)]Moments $\phi,\psi$ of $A$ are congruent if and only if their active parts $\phi_{act},\psi_{act}$ are isomorphic under $A$.\end{itemize}\end{lma}

\begin{proof}(i) Consider split moments $\phi=ir$ and $\psi=js$. If $i=jj'$ then $\phi=ir=jj'r$ and hence $\psi\phi=\phi$. Conversely, if $\psi\phi=\phi$ then by (m2) and (m3) $j_*s_*(\phi)=\phi$ and hence, by Lemma \ref{formal}(v), $js_*(\phi)s=\phi$. Splitting the moment $s_*(\phi)$ as $j's'$ we get $jj's's=ir$ and, since the last identity represents two splittings of the same moment, we can assume without loss of generality that $r=s's$ and $i=jj'$.

(ii) The second assertion follows from the first and the definition of the partial orders. For the first observe that $$f_*(\phi\psi)\overset{(\ref{formal}iii)}{=}(f\phi)_*(\psi)\overset{(m4)}{=}(f_*(\phi)f)_*(\psi)\overset{(m3)}{=}f_*(\phi)_*(f_*(\psi))\overset{(m2)}{=}f_*(\phi)f_*(\psi).$$

(iii) If $\phi=\phi\psi\phi=\phi\psi$ and $\psi=\psi\phi\psi=\psi\phi$ then $\phi_{act}\psi_{in}\psi_{act}=\phi_{act}$ and $\psi_{act}\phi_{in}\phi_{act}=\psi_{act}$. This implies that $\phi_{act}\psi_{in}$ and $\psi_{act}\phi_{in}$ induce inverse isomorphisms between $\psi_{act}$ and $\phi_{act}$.

Conversely, if $\rho\psi_{act}=\phi_{act}$ then $\phi\psi=\phi_{in}\phi_{act}\psi_{in}\psi_{act}=\phi_{in}\rho\psi_{act}\psi_{in}\psi_{act}=\phi$; dually, if $\sg\phi_{act}=\psi_{act}$ then $\psi\phi=\psi$.\end{proof}

\begin{rmk}\label{LRB}According to Lemma \ref{formal}i the moments of any object of a moment category form a submonoid of the endomorphism monoid and fulfill the \emph{Sch\"utzenberger relation} $\phi\psi\phi=\phi\psi$. Such monoids are known in semigroup literature as \emph{left regular bands}, cf. \cite{MSS}. More precisely, a \emph{band} is a semigroup consisting of idempotent elements, and a band is said to be \emph{left regular} if the Sch\"utzenberger relation holds. In contrast to the semigroup literature, we shall always assume that a left regular band has a neutral element, i.e. is a monoid. In the presence of a neutral element, the idempotency of the elements follows from the Sch\"utzenberger relation. A \emph{morphism of left regular bands} is required to preserve the multiplicative structure, but not necessarily the neutral element.  This is important for us because pushforward is a morphism of left regular bands by Lemma \ref{Greenorder}ii, but \emph{not} in general a morphism of monoids. Pushforward $f_*:m_A\to m_B$ preserves the neutral element if and only if $f:A\to B$ is \emph{active}, cf. Definition \ref{formaldfn} and Proposition \ref{formalsplit}.

In a category with moment structure each object has thus a \emph{local monoid} $m_A$ of moments which is a left regular band. The partial order relation on $m_A$ (defined in Definition \ref{geometric}) is known in literature as Green's $\Ll$-relation. The congruence relation is known as Green's $\Dd$-relation, and it is well-known that for any left regular band the quotient by the $\Dd$-relation defines its \emph{universal commutative quotient}. In particular, the local monoid $m_A$ is commutative if and only if its congruence relation is discrete, i.e. all moments of $A$ are centric, cf. Definition \ref{geometric}.

As we will see in Propositions \ref{centric} and \ref{Burnside} below, this generalises to moment categories: a moment category satisfies centricity axiom (MC) if and only if all its moments are centric if and only if moments of the same object commute. Moreover, every moment category admits a universal centric quotient.

In general, the quotient $m_A/\!\simeq$ by the congruence relation is thus a commutative band. Commutative bands are also known as meet-semilattices if multiplication is viewed as meet operation. The congruence class of $1_A$ serves as top element for the partial order. If the local monoid $m_A$ is finite (which will always be the case for us) we get a finite meet-semilattice $m_A/\!\simeq$ with top element $1$, and again it is well-known that one can define a join operation $x\vee y$ by taking the meet of all $z$ such that $x\leq z$ and $y\leq z$. We get in this way a \emph{lattice} with bottom element $0$. Since the congruence class of $1_A$ is singleton, we get a lattice with $0\not=1$ as soon as $A$ has non-identity moments.\end{rmk}

We are grateful to Steve Lack for having pointed out to us the following result which is an important tool for constructing moment categories.

\begin{prp}\label{completion}Every category $\CC$ with moment structure admits an idempotent completion into a moment category $\overline{\CC}$ whose objects are the moments of $\CC$. Congruent moments in $\CC$ give rise to isomorphic objects in $\overline{\CC}$.\end{prp}

\begin{proof}By definition, the objects of $\overline{\CC}$ are the moments of $\CC$, and for any pair of moments $(\phi,\psi)\in m_A\times m_B$, the morphism set is defined by$$\overline{\CC}(\phi,\psi)=\{f\in\CC(A,B)\,|\,f_*(\phi)\leq\psi\}.$$The identity of $\overline{\CC}(\phi,\phi)$ is $\phi$, where composition in $\overline{\CC}$ is defined like in $\CC$. If $f_*(\phi)\leq \psi$ and $g_*(\psi)\leq\zeta$ then $(gf)_*(\phi)\leq\zeta$ so that we get indeed a category in this way.

We now apply Proposition \ref{formalsplit} in order to show that $\overline{\CC}$ is a moment category where we define the moment set $m_\phi$ of an object $\phi\in m_A$ by$$m_\phi=\{\psi\in m_A\,|\,\psi\leq\phi\}$$so that $\CC$ fully embeds into $\overline{\CC}$ under preservation of the moment structures.

The pushforward operation $f_*:m_\phi\to m_\psi$ for any $f:\phi\to\psi$ in $\overline{\CC}$ is defined by restricting the pushforward operation $f_*:m_A\to m_B$ of $\CC$ because it follows from $f_*(\phi)\leq\psi$ that $f_*$ takes a moment in $m_\phi$ to a moment in $m_\psi$. These restricted pushforward operations fulfill the axioms (m1)-(m4) of a moment structure on $\overline{\CC}$.

It remains to be shown that the moments of $\overline{\CC}$ are split. Given $\psi\in m_\phi$, the moment $\psi$ splits over the object $\psi\phi$ in $\overline{\CC}$. Indeed, $\psi:\phi\to\psi\phi$ is a retraction with section $\psi:\psi\phi\to\phi$ in $\overline{\CC}$ whose associated moment is $\psi:\phi\to\phi$.

For the second assertion observe that $\psi\phi:\phi\to\psi$ and $\phi\psi:\psi\to\phi$ are mutually inverse morphisms in $\overline{\CC}$ if and only if $\phi$ and $\psi$ are congruent moments in $\CC$.\end{proof}

\begin{exms}\label{examples}The following examples of moment categories are all locally finite with a countable or finite set of objects. These categories play important roles in algebraic topology or algebraic combinatorics. It is somehow surprising that they share the feature of carrying a moment structure.\vspace{1ex}

(a) \emph{The category $\Gamma$ of Segal}\vspace{1ex}

Segal's category $\Gamma$ \cite{Se} is the category of finite sets $\nn=\{1,\dots,n\}$ (where $\underline{0}$ denotes the empty set) with operators $\mm\to\nn$ given by ordered $m$-tuples of pairwise disjoint subsets of $\nn$. Composition is defined by\begin{diagram}(\kk&\rTo^{(M_1,\dots,M_k)}&\mm&\rTo^{(N_1,\dots,N_m)}&\nn)=(\kk&\rTo^{(\bigcup_{j_1\in M_1}N_{j_1},\dots,\bigcup_{j_k\in M_k}N_{j_k})}&\nn).\end{diagram}
We now define the following active/inert factorisation system on $\Gamma$:

an operator $f=(N_1,\dots,N_m):\mm\to\nn$ is \emph{active} if $\nn=\bigsqcup\limits_{i=1}^m N_i$.

an operator $f=(N_1,\dots,N_m):\mm\to\nn$ is \emph{inert} if each $N_i$ is singleton.\vspace{1ex}

\noindent In particular, inert morphisms correspond to injections $\mm\to\nn$ and active morphisms can be considered as partitions of the target indexed by the elements of the source. It is now straightforward to check that this defines a factorisation system on $\Gamma$ fulfilling the axioms (M1), (M2), (MC) of a centric moment category. For instance, axiom (MC) holds since active morphisms have sections if and only if the associated partition only contains singletons and empty subsets of the target.

It is well-known that the dual category $\Gamma^\op$ may be identified with a skeleton of the category of finite based sets and base-point preserving maps. Dualising the active/inert factorisation system on $\Gamma$ induces an inert/active factorisation system on $\Gamma^\op$ which has extensively been used by Lurie \cite{Lu} in his theory of $\infty$-operads. We borrowed our terminology from him, adding just the extra-term ``moment''.

The dual $\Gamma^\op$ can also be identified with a skeleton of the category of finite sets and \emph{partial} maps between them. In this setting, the inert/active factorisation of a partial map is its factoriation into a partial identity followed by a total map. The inclusion of the active part $\Gamma_{act}^\op$ into $\Gamma^\op$ can be interpreted as the inclusion of the category of finite sets and \emph{total} maps into the category of finite sets and \emph{partial} maps. This interpretation illustrates well the dual of Proposition \ref{corestriction}.

We will see that the moment category $\Gamma$ plays an important universal role insofar as every \emph{locally finite, unital} moment category $\CC$ comes equipped with an essentially unique augmentation $\gamma_\CC:\CC\to\Gamma.$ This augmentation takes an object of $\CC$ to the set of its elementary moments and views the morphisms of $\CC$ as ``partial partitions'' of their target, cf. Proposition \ref{augmentation}.\vspace{1ex}

(b) \emph{The simplex category $\Delta$}\vspace{1ex}

The \emph{simplex category} $\Delta$ is the category of finite non-empty ordinals $[m]=\{0,\dots,m\},\,m\geq 0,$ and order-preserving maps. A simplicial operator $f:[m]\to[n]$ is \emph{active} if it is \emph{endpoint-preserving}, i.e. $f(0)=0$ and $f(m)=n$, and \emph{inert} if it is \emph{distance-preserving}, i.e. $f(i+1)=f(i)+1$ for $i=0,\dots,m-1$.

It is straightforward to check that the axioms (M1), (M2) of a moment category are satisfied. Axiom (MC) does not hold: any map $[n]\to[0]$ is active but has several inert sections whenever $n>0$. In general, an active map $f:[n]\rAct\, [m]$ has an inert section $g:[m]\rIn\, [n]$ if and only if there is a subinterval $[s,s+m]\subset[0,n]$ such that $f(0)=f(s)$ and $f(s+m)=f(n)$ while $f$ restricted to $[s,s+m]$ is injective. If $m>0$, then the section $g$ is uniquely determined by $f$. Therefore, by Lemma \ref{centricmoment} below, a moment of $[n]\not=[0]$ is centric if and only if it splits over $[m]\not=[0]$. The monoid structure of the moments of $[n]$ can be interpreted geometrically in terms of the corresponding inert subobjects, i.e. subintervals of $[n]$. The product of two moments is commutative if the two corresponding inert subobjects have an intersection in $\Delta$ in which case the product represents this intersection. Otherwise, the product is non-commutative. For instance, for the moments $\phi=(0,1,1,1)$ and $\psi=(2,2,2,3)$ of $[3]$ we get $\phi\psi=(1,1,1,1)\not=(2,2,2,2)=\psi\phi$.

Note that $\Delta$ also admits the familiar epi/mono factorisation system where the epimorphisms are called degeneracy operators and the monomorphisms face operators. Every degeneracy operator is retractive (cf. Lemma \ref{formal}iv) while active (resp. inert) face operators are usually called \emph{inner} (resp. \emph{outer}) face operators.\vspace{1ex}

(c) \emph{The categories $\Theta_n$ of Joyal}\vspace{1ex}

The categories $\Theta_n$ of Joyal \cite{J} have several equivalent definitions. For $n=1$ we recover example (b) since $\Theta_1=\Delta$. For general $n>0$, the objects of $\Theta_n$ are $n$-level trees. The maps can be described by first taking an $n$-level tree $T$ to an $n$-graph $T_*$, and then applying the free (strict) $n$-category functor $F_n$ from $n$-graphs to $n$-categories. This leads to the following definition (cf. \cite{Be})$$\Theta_n(S,T)=\nCat(F_n(S_*),F_n(T_*)).$$ An alternative way of defining $\Theta_n$ is as an iterated wreath product of $\Delta$ (cf. \cite{Be2}). We will see below (cf. Proposition \ref{unitalwreath}) that for any two unital moment categories $\CC,\,\DD$ there is a well-defined wreath product $\CC\wr\DD$ which is again a unital moment category. Since $\Delta$ is a unital moment category, this permits to define iterated wreath products of $\Delta$, and it turns out that $\Theta_n$ is an $n$-fold wreath product $\Delta\wr\cdots\wr\Delta$ yielding the moment structure of $\Theta_n$ for free.

The active (resp. inert) morphisms of $\Theta_n(S,T)$ have a geometric interpretation in terms of the $n$-level tree structures of $S$ and $T$. The active morphisms correspond to tree-partitions of $T$ labelled by $S$. The inert morphisms correspond to immersions of $S$ as a plain subtree of $T$. The inert morphisms are those belonging to the image of the globular maps $S_*\to T_*$ under the free functor $F_n$. The existence of this active/inert factorisation system of $\Theta_n$ has been established in \cite[Lemma 1.11]{Be}.

Note that $\Theta_n$ also admits the familiar epi/mono factorisation system where the epimorphisms are called degeneracy operators and the monomorphisms face operators. Every degeneracy operator is retractive while the moment structure induces a natural notion of inner (i.e. active) and outer (i.e. inert) face operator in $\Theta_n$.\vspace{1ex}

(d) \emph{Idempotent completion of a left regular band}\label{lrb}\vspace{1ex}

We refer the reader to Margolis-Saliola-Steinberg \cite[Section 2]{MSS} for more details concerning the theory of left regular bands. Each left regular band $L$ (cf. Remark \ref{LRB}) can be considered as a one-object category $\CC_L$ with endomorphism monoid the given left regular band $L$. If we define the unique moment set also to be equal to $L$, and the pushforward operation to be left translation, then all axioms of a moment structure are satisfied, axiom (m4) being the Sch\"utzenberger relation. Proposition \ref{completion} applies and $\CC_L$ fully embeds into a moment category $\overline{\CC}_L$, whose objects are the elements of $L$. Moreover, elements in $L$ are \emph{congruent} if and only if they are \emph{isomorphic} as objects of $\overline{\CC}_L$. Therefore, a \emph{skeleton} of $\overline{\CC}_L$ is itself a moment category $\widehat{\CC}_L$ whose objects are the elements of the quotient $L/\!\simeq$.

These moment categories $\overline{\CC}_L$ constructed out of left regular bands $L$ have very special properties: each active morphism is retractive, and the (retr)active part of $\overline{\CC}_L$ is opposite to the poset underlying $L$. The inert subobjects of an object of $\overline{\CC}_L$ are in bijection with the subobjects of the corresponding element of $L$.

The most prominent example of a left regular band is the \emph{face monoid} $\Ll_\Aa$ of a \emph{hyperplane arrangement} $\Aa$ in Euclidean space, cf. \cite{MSS}. In this case, the quotient $\Ll_\Aa/\!\simeq$ by the congruence relation is the so-called \emph{intersection lattice} $\Ii_\Aa$ of $\Aa$. The elements of the face monoid (the ``facets'') can be identified with non-empty intersections of the half-spaces delimited by the hyperplanes of $\Aa$. The product $z=xy$ in $\Ll_\Aa$ of two facets is the first facet $z$ crossed by a segment joining an interior point of $x$ to an interior point of $y$. Two facets are congruent if and only if they generate the same linear support in the intersection lattice $\Ii_\Aa$.

As illustration, let us consider the Coxeter arrangement $\Aa_{\Sg_3}$ of the symmetric group $\Sg_3$ on three letters. This is the so-called \emph{braid arrangement} on three strands. The facets of the face monoid $\Ll_{\Aa_{\Sg_3}}$ are in bijection with left cosets of the standard Coxeter subgroups of $\Sg_3$ (namely $\Sg_1\times\Sg_1\times\Sg_1,\Sg_1\times\Sg_2,\Sg_2\times\Sg_1$ and $\Sg_3$) where the underlying poset of the left regular band $\Ll_{\Aa_{\Sg_3}}$ is reverse inclusion of left cosets. The objects of the intersection lattice $\Ii_{\Aa_{\Sg_3}}$ are conjugacy classes of these standard Coxeter subgroups in $\Sg_3$. This yields the following inert part of $\widehat{\CC}_{\Ll_{\Aa_{\Sg_3}}}$

$$
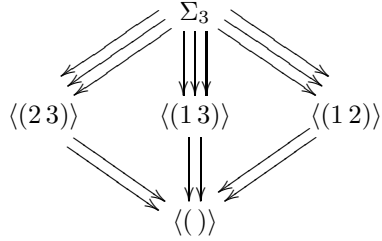
\xymatrix{&\Sg_3\ar[dr]\ar@<1ex>[dr]\ar@<2ex>[dr]\ar[d]\ar@<-1ex>[d]\ar@<1ex>[d]&\\
\langle(2\,3)\rangle\ar@{<-}[ru]\ar@<1ex>@{<-}[ru]\ar@<2ex>@{<-}[ru]&\langle(1\,3)\rangle\ar@<-0.5ex>[d]\ar@<0.5ex>[d]&\langle(1\,2)\rangle\\
&\langle(\,)\rangle\ar@<1ex>@{<-}[lu]\ar@{<-}[lu]\ar@{<-}[ru]\ar@<1ex>@{<-}[ru]&}$$

\noindent where the inert arrows correspond to left cosets (of the target inside the source) with obvious composition law. This inert part of  $\widehat{\CC}_{\Ll_{\Aa_{\Sg_3}}}$ contains three subcategories whose topological realisations are classifying spaces for the \emph{braid group} $B_3$ on three strands, cf. \cite[Remark 2.7]{Be0}.

\end{exms}

The following proposition was suggested to us by Richard Garner. It shows that the category $\LRB$ of left regular bands (with neutral element) and morphisms of left regular bands is an example of a \emph{large} moment category.

\begin{prp}\label{large}The category $\LRB$ is a moment category with active (resp. inert) morphisms the monoid morphisms (resp. the order-ideal inclusions).

Each moment category $\CC$ is endowed with a moment functor $m_\CC:\CC\to\LRB$ taking objects to their local moment monoids, and morphisms to the induced pushforward operation.\end{prp}

\begin{proof}Each morphism of left regular bands $f:M\to N$ factors as a monoid morphism $f_{act}:M\to f(1)N$ followed by an order-ideal inclusion $f_{in}:f(1)N\to N$. This factorisation is essentially unique, because any order-ideal inclusion $xN\to N$ admits a monoid retraction $N\to xN$ given by left translation by $x$; observe that the Sch\"utzenberger relation implies $x(yz)=(xy)(xz)$, i.e. left translation preserves the multiplicative structure. The uniqueness of this retraction yields axiom (M1) of a moment category. For (M2) observe that if a monoid morphism $f:M\to N$ remains a monoid morphism when restricted to the order-ideal $i:xM\to M$ then $f(x)=1$, therefore $fir=f$ where $r:M\to xM$ is left translation by $x$.

The second assertion follows from Lemmas \ref{formal}i and \ref{Greenorder}ii and (m1), (m3).\end{proof}

\subsection{Centricity}We now discuss in more detail centricity axiom (MC) and its relationship with corestriction structures of Cockett-Lack \cite{CL1}.

\begin{lma}\label{centricmoment}A moment is centric if and only if its inert part is the only inert section of its active part.\end{lma}

\begin{proof}For inert sections $i,i'$ of an active morphism $r$, the moments $ir$ and $i'r$ are congruent. Therefore, if $\phi=\phi_{act}\phi_{in}$ is centric then $\phi_{in}$ is the only inert section of $\phi_{act}$. Conversely, if the latter holds then for any congruence $\phi\simeq\psi$, Lemma \ref{Greenorder}iii yields $\rho$ such that $\rho\phi_{act}=\psi_{act}$ so that $\phi_{in}=\psi_{in}\rho$ whence $\phi=\psi_{in}\rho\phi_{act}=\psi$.\end{proof}

\begin{prp}\label{centric}For a moment category the following conditions are equivalent:

\begin{enumerate}\item[(i)]the moment structure is centric;\item[(ii)]moments of the same object commute;\item[(iii)]centricity axiom (MC) holds.\end{enumerate}\end{prp}

\begin{proof}We have seen in Remark \ref{LRB} that (i) implies (ii). Assume now (ii) and consider an active morphism $f:A\to B$ with inert sections $i,j:B\to A$. Then $jfif=ifjf$ implies $jf=if$ and hence $i=j$, whence (iii). By Lemma \ref{centricmoment}, (iii) implies (i).\end{proof}

\begin{prp}\label{corestriction}\emph{Corestriction structures} in the sense of Cockett-Lack \cite{CL1} correspond one-to-one to \emph{centric moment structures}. In particular, split corestriction categories are the same as centric moment categories.\end{prp}

\begin{proof}A corestriction structure is the dual of a restriction structure (cf. \cite[2.1.1]{CL1}). It is defined in terms of so-called \emph{cocombinators} $\CC(A,B)\to\CC(B,B):f\mapsto f_*(1_A).$ These cocombinators extend to pushforward operations $f_*:m_A\to m_B$ by the rule $f_*(\phi)=(f\phi)_*(1_A)$ where $m_A=\{\phi\in\CC(A,A)\,|\,\phi_*(1_A)=\phi\}$. Calling the elements of $m_A$ moments, the axioms of Cockett-Lack can be stated as follows:
\begin{itemize}\item[(C1)]$f_*(1)f=f$ for any morphism $f$;
\item[(C2)]$\phi\psi=\psi\phi$ for any moments $\phi,\psi$ of the same object;
\item[(C3)]$(\psi f)_*(1)=\psi\, f_*(1)$ for any morphism $f$ and moment $\psi$ of the target of $f$;
\item[(C4)]$g\, f_*(1)=(gf)_*(1)\,g$ for any composable morphisms $f$ and $g$.\end{itemize}

In \cite[Lemma 2.1iii]{CL1} Cockett and Lack deduce from (C1),(C2),(C3),(C4) that $g_*(\phi)=(g\phi)_*(1)$ for any moment $\phi$ of the source of $g$. This implies axiom (m3). Axiom (C4) then yields $g\phi=g_*(\phi)g$ which is (m4). It follows from (C3) that the composite of two moments is a moment, which is equivalent to (m2), as soon as (m3) holds. Finally, (C1) implies (m1). By (C2), moments of the same object commute. Conversely, in a centric moment structure axiom (C2) holds, (C4) follows from (m4), (C1) follows from (C4) and (m1), while (C3) follows from (m2), (m3) and Lemma \ref{formal}iii. A corestriction category is split \cite[2.3.3]{CL1} if and only if all moments split. Propositions \ref{formalsplit} and \ref{centric} thus establish the second statement.\end{proof}

\begin{rmk}\label{relationship}This close relationship between corestriction and moment structures was a surprise for us. However, in most examples relevant to operads, the moment categories are \emph{not} centric, but the \emph{splittings} are needed for the definition of the operad-type associated to a unital moment category. Therefore, starting from centric moment categories (i.e. split corectriction categories) our main concern consisted in dropping centricity (while keeping the splittings), whereas the main concern of Cockett-Lack was to drop the existence of splittings (while keeping centricity). There are nevertheless similarities between both approaches.

We have seen that every category with moment structure admits an \emph{idempotent completion} turning it into a moment category. The existence of this idempotent completion relies (just as in the case of corestriction categories) on the completely equational character of the axioms (m1)-(m4) of a moment structure.

Our terminology differs from that of Cockett-Lack, partly in order to avoid the use of too many co-s. The dictionnary is as follows:\begin{center}moment=corestriction idempotent, active=cototal, retractive=split corestriction.\end{center}\end{rmk}

We end this introductory section by showing that each moment category has a universal centric quotient.

\begin{lma}\label{retractive}For any moment category, the pushout of a retractive morphism along an active morphism exists in the active part and is again retractive.\end{lma}

\begin{proof}Consider the following diagram

$$\xymatrix{A\ar[r]^f|+\ar[d]^r|+&B\ar[d]^{r'}|+\ar[rdd]^g&\\A'\ar@{ >->}@<1ex>[u]^i\ar[r]_{f'}|+\ar[rrd]_{g'}&B'\ar@{ >->}@<1ex>[u]^{i'}\ar@{.>}[rd]|h&\\&&C}$$

\noindent in which $i$ is any inert section of the retractive morphism $r$, and $i'r'=f_*(ir)$. We claim that the retractive morphism $r':B\rAct B'$ is the pushout in $\CC_{act}$ of $r$ along the active morphism $f:A\rAct B$. For this, we check the universal property and choose two active morphisms $g,g'$ such that $g'r=gf$. If $h:B'\rAct C$ with $hr'=g$ exists then necessarily $h=gi'$. So there is no choice for $h$, and it remains to be shown that $hf'=g'$, that $h$ is active, and that $hr'=g$.

Indeed, $hf'=gi'f'=gfi=g'ri=g'$. Since $g'$ and $f'$ are active and $hf'=g'$, it follows from general properties of factorisation systems that $h$ is active as well. Therefore, $g$ and $gi'=h$ are active, so that by axiom (M2), $g=hr'$ as required.\end{proof}

For a moment category $\CC$ we define $B\CC$ to be the category with same objects as $\CC$ but with morphism-set $(B\CC)(A,B)$ the set of isomorphism classes of cospans $$(f,r)=(A\rAct^f B'\lAct^r B)$$ where $f$ is active and $r$ is retractive. Thanks to Lemma \ref{retractive}, these isomorphism classes compose in the following way: $(g,s)(f,r)=(gf',s'r)$ where $f'$ (resp. $s'$) is the pushout of $f$ (resp. $s$) along $s$ (resp. $f$) in the active part of $\CC$.

We thus get an identity on objects functor $\CC\to B\CC$ taking the morphism $f:A\to B$ to the cospan $(f_{act},r_f)$ where $r_f$ is the unique active retraction of $f_{in}$.

Two parallel morphisms $f,g:A\rightrightarrows B$ are said to be \emph{congruent} (denoted $f\simeq g$) if they are identified under $\CC\to B\CC$. By Lemma \ref{Greenorder}iii, this is the case if and only if the moments $f_*(1_A)$ and $g_*(1_A)$ are congruent in $m_B$.

Observe that $B\CC$ comes equipped with an active/inert factorisation system for which isomorphism classes of cospans of the form $(f_{act},1)$ are active, and isomorphism classes of cospans of the form $(1,r_f)$ are inert. It is straightforward to verify that $B\CC$ satisfies the axioms (M1), (M2), (MC) of a centric moment category.

\begin{prp}\label{Burnside}A moment category $\CC$ is centric if and only if the functor $\CC\to B\CC$ is invertible. In general, the functor $\CC\to B\CC$ is initial among moment functors out of $\,\CC$ taking values in centric moment categories.\end{prp}

\begin{proof}The first statement is a consequence of the fact that each moment $\phi$ of an object $A$ equals $\phi_*(1_A)$. For the second statement, let $F:\CC\to\DD$ be a functor with centric target $\DD$ and consider the following commutative diagram of functors\begin{diagram}[small]\CC&\rTo^F&\DD\\\dTo&&\dTo_\cong\\\CC/\!\!\simeq&\rTo&\DD/\!\!\simeq\end{diagram}in which the right vertical functor is an isomorphism because $\DD$ is centric. We therefore get the required factorisation  $F:\CC\to\CC/\!\!\simeq\,\to\DD$ which is unique because the functor $\CC\to\CC/\!\!\simeq$ is identity on objects and full.\end{proof}

\begin{rmk}It follows from Lemma \ref{Greenorder}iii that the categorical congruence restricts to Green's $\Dd$-relation on $m_A$ for each object $A$. For congruent parallel morphisms $f,g:A\rightrightarrows B$, the pushforwards $f_*(\phi)$ and $g_*(\phi)$ are congruent in $m_B$ for any $\phi\in m_A$. The first part of Proposition \ref{Burnside} is a \emph{Representation Theorem} for centric moment categories. Indeed, any centric moment category $\CC$ is isomorphic to $B\CC$, and is thus entirely determined by its active part.

The dual representation theorem for \emph{split restriction categories} (cf. Remark \ref{relationship}) has been obtained by Cockett-Lack using different methods, cf. \cite[Theorem 3.4]{CL1}.\end{rmk}

\section{Unital moment categories}\label{unitalmomentsection}

The purpose of this section is to single out a class of moment categories, called \emph{unital}, in which moments decompose in a canonical way into \emph{elementary} moments.  A moment category is unital if there is a sufficient supply of so-called \emph{unit objects}. The elementary moments are then those which split over a unit object.

Each \emph{active} morphism $f:A\to B$ induces a \emph{partition} of the identity moment of $B$ into moments $f_*(\alpha)$ where $\alpha$ runs through the elementary moments of $A$. This will allow us to associate, with every unital moment category of finite type, a specific \emph{operad type} incorporating these partitions as its \emph{substitutional structure}. We shall call operads of type $\CC$ simply $\CC$-operads.

The unital moment category $\Gamma$ enjoys a universal property. Each unital moment category of finite type comes equipped with an essentially unique augmentation $\gamma_\CC:\CC\to\Gamma$, and $\Gamma$-operads turn out to be precisely \emph{symmetric operads} in the sense of Boardman-Vogt \cite{BV} and May \cite{May}. The augmentation $\gamma_\CC$ induces an adjunction between the categories of $\CC$-operads and of $\Gamma$-operads so that $\CC$-operads can be \emph{symmetrised}.

The \emph{wreath product of $\,\Gamma$-augmented categories} as defined in \cite{Be2} takes a pair $(\CC,\DD)$ of unital moment categories to a unital moment category $\CC\wr\DD$. We give here an intrinsic description of this wreath product not depending on a chosen augmentation over $\Gamma$ and valid without finite type hypothesis.

It follows from \cite{Be2} that the operator categories $\Theta_n$ of Joyal \cite{J} are iterated wreath products of the simplex category $\Delta$. The unital moment structure on $\Delta$ induces thus a unital moment structure on $\Theta_n$. It turns out that the resulting $\Theta_n$-operads are precisely the \emph{$(n-1)$-terminal $n$-operads} of Batanin \cite{Ba,Ba2} and that our symmetrisation functor coincides with Batanin's \cite{Ba2} in this special case.\vspace{1ex}

\subsection{Units}

We shall say that an object $U$ of a moment category is \emph{primitive} provided $U$ has non-identity moments, but they all are congruent.

It amounts to the same to require that the universal semilattice quotient $m_U/\!\simeq$ is the two-element lattice $\{0,1\}$. Indeed, the congruence class of $1$ is singleton and contains the identity-moment $1_U$ while the congruence class of $0$ contains all the other moments of $U$, cf. Remark \ref{LRB}.

\begin{dfn}\label{unitaldefinition}An object $\,U$ of a moment category is called a \emph{unit} if\begin{itemize}\item[(U1)]the object $U$ is primitive;\item[(U2)]each active morphism with target $\,U$ has one and only one inert section.\end{itemize}

A moment category is called \emph{unital} if for each object $A$ there exists an essentially unique pair $(U,\phi)$ consisting of a unit $\,U$ and an active map $\phi:U\rAct A$.\vspace{1ex}

A moment is called \emph{elementary} if it splits over a unit.\vspace{1ex}

Two moments of the same object are called \emph{disjoint} if they do not share common elementary submoments.\vspace{1ex}

The \emph{cardinality} of an object is the cardinality the set of its \emph{elementary} moments. A \emph{nilobject} is an object of cardinality $0$.\vspace{1ex}

A unital moment category is \emph{of finite type} if the underlying category is small and every object has finite cardinality.\end{dfn}

\begin{lma}\label{incomparable}--

\begin{enumerate}\item[(i)]Any inert (resp. active) morphism between units is invertible;\item[(ii)]Elementary moments of the same object are either equal or disjoint.\end{enumerate}\end{lma}
\begin{proof}(i) Both assertions are equivalent. Let $i:U\rIn U'$ be an inert morphism between units. Then by (U2) $i$ is the unique inert section of the active retraction $r:U'\rAct U$. By Lemma \ref{centricmoment} this implies that the moment $ir$ of $U'$ is centric so that by (U1) we get $ir=1_{U'}$.

(ii) It suffices to show that comparable elementary moments are equal. By Lemma \ref{Greenorder}i, moments are comparable if and only if the associated inert subobjects are comparable. For elementary moments these subobjects are units so that (i) allows us to conclude.\end{proof}

\begin{rmk}\label{unitalexamples}All examples of Section \ref{examples} are unital moment categories:\vspace{1ex}

(a) Segal's category $\Gamma$ has a single unit, the one-element set $\11$. The elementary moments of $\nn$ correspond to inert subsets $\11\rIn\nn$, i.e. to elements of the set $\nn$. The cardinality of $\nn$ is thus $n$ and the only nilobject of $\Gamma$ is $\underline{0}$. For every objet $\nn$ there is a \emph{unique} active morphism $\11\rAct\nn$.

(b) The simplex category $\Delta$ has a single unit, the segment $[1]$. The elementary moments of $[n]$ correspond to inert subobjects $[1]\rIn\, [n]$, i.e. to subsegments. The cardinality of $[n]$ is thus $n$ and $[0]$ is the only nilobject of $\Delta$. For every object $[n]$ there is a \emph{unique} active morphism $[1]\rAct\, [n]$.

(c) Joyal's category $\Theta_n$ has a single unit, the linear $n$-level tree $U_n$ of height $n$. The elementary moments of an $n$-level tree $T$ correspond to the vertices of $T$ of height $n$. The nilobjects of $\Theta_n$ are thus the $n$-level trees of height $<n$. Again, each $n$-level tree $T$ receives a unique active morphism $U_n\rAct T$, cf. Proposition \ref{unitalwreath}.

(d) The idempotent completion $\overline{\CC}_M$ of a left regular band $M$ has a single unit, represented by the neutral element $1\in M$. Every morphism in $\widehat{\CC}_M$ is either inert or retractive or a moment, cf. Remark \ref{LRB}. All objects different from $1$ are nilobjects.\end{rmk}

From now on we fix a \emph{unital moment category} $\CC$. The following lemma is the main reason for having introduced the notion of elementary moment. The argument is roughly speaking dual to the one establishing that for any set-mapping $f:X\to Y$ and $x\in X$ there is a unique $y\in Y$ such that $x\in f^{-1}(y)$.

\begin{lma}\label{active=surj}For any elementary moment $\psi$ of $B$ and active morphism $f:A\rAct B$ there is a unique elementary moment $\phi$ of $A$ such that $\psi\leq f_*(\phi)$.\end{lma}
\begin{proof}The elementary moment $\psi$ factors through $\psi_{act}:B\rAct U$ where $U$ is a unit. By (U2) the composite active morphism $\psi_{act}f:A\rAct U$ has an inert section, thus defining an elementary moment $\phi$ of $A$. By Lemma \ref{retractive}, the moment $f_*(\phi)$ is associated with the pushout along $f$ of $\psi_{act}f$ in $\CC_{act}$, whence $\psi\leq f_*(\phi)$.

Assume we are given two elementary moments $\phi,\phi'$ of $A$ such that $\psi\leq f_*(\phi)$ and $\psi\leq f_*(\phi')$. In virtue of Lemma \ref{retractive}, $\psi$ is then an elementary moment of the pushout $Q$ of $f_*(\phi)_{act}$ and $f_*(\phi')_{act}$ in the active part of $\CC$:
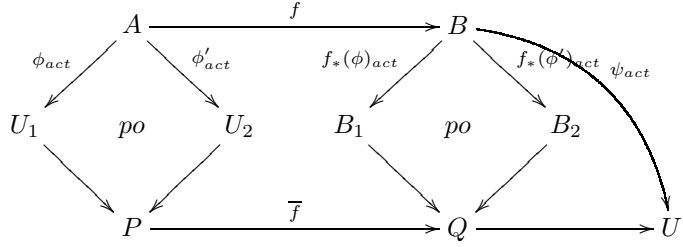
 $$\xymatrix{&A\ar[ld]_{\phi_{act}}\ar[rd]^{\phi'_{act}}\ar[rrr]^f&&&B\ar[ld]_{f_*(\phi)_{act}}\ar[rd]^{f_*(\phi')_{act}}\ar@/^5ex/[rrdd]^{\psi_{act}}&&\\
U_1\ar[rd]&po&U_2\ar[ld]&B_1\ar[rd]&po&B_2\ar[ld]&\\&P\ar[rrr]^{\overline{f}}&&&Q\ar[rr]&&U}$$\noindent By a diagram chase, $Q$ can be identified with the pushout along $f$ of the pushout $P$ of $\phi_{act}$ and $\phi'_{act}$. We thus get an active morphism $\overline{f}:P\rAct Q$ whose target has an elementary moment. It follows then that $P$ has as well an elementary moment so that $\phi$ and $\phi'$ are not disjoint. According to Lemma \ref{incomparable}ii this implies $\phi=\phi'$.\end{proof}

\begin{prp}\label{inducedpartition}Pushforward along any morphism $f:A\to B$ takes disjoint moments of $A$ to disjoint moments of $B$. If $f$ is active, then $f_*$ induces a partition of the set of elementary moments of $B$, indexed by the elementary moments of $A$.\end{prp}

\begin{proof}For the first assertion, we can assume that $f$ is either inert or active. Two moments are disjoint precisely when their associated inert subobjects do not share an inert subobject with unital domain $U$. If $f$ is inert, the pushforward operation $f_*$ associates to an inert subobject $\phi_{in}:A'\rIn A$ the inert subobject $\phi_{in}f:A'\rIn A\rIn B$, and hence pushforward along inert morphisms preserves disjointness. If $f$ is active, Lemma \ref{active=surj} shows that distinct elementary moments of $A$ are taken to disjoint moments of $B$. This suffices by Lemma \ref{incomparable}ii.

For the second assertion we associate with each elementary moment $e_\alpha$ of $A$ the set of elementary moments $e_\beta$ of $B$ such that $e_\beta\leq f_*(e_\alpha)$. According to Lemmas \ref{incomparable} and \ref{active=surj} this defines a partition of the set of elementary moments of $B$.\end{proof}

Let us illustrate Proposition \ref{inducedpartition} by means of the examples \ref{examples}a-b. In $\Gamma$, an active morphism $f:\mm\rAct\nn$ can be analysed by considering the pushforwards of the elementary moments of $\mm$. The latter correspond to singleton subsets of $\mm$ while their pushforwards along $f$ correspond to subsets of $\nn$. Two subsets are then ``disjoint'' in the sense of Definition \ref{unitaldefinition} precisely when they meet in the nilobject $\underline{0}$ which recovers the usual meaning of disjointness. According to Lemma \ref{active=surj}, the subsets of $\nn$ obtained as pushforwads along $f:\mm\rAct\,\nn$ of the singleton subsets of $\mm$ are mutually disjoint, and cover the singleton subsets of $\nn$.

In the same manner, an active morphism $f:[m]\rAct\,[n]$ in $\Delta$ can be analysed by considering the pushforwards of the elementary moments of $[m]$. The latter correspond to subsegments of $[m]$ while their pushforwards along $f$ correspond to subintervals of $[n]$. Two subintervals are then ``disjoint'' in the sense of Definition \ref{unitaldefinition} precisely when they meet in the nilobject $[0]$ (i.e. when they share just an extremity) or when they do not intersect. Again, according to Lemma \ref{active=surj}, the subintervals of $[n]$ obtained as pushforwards along $f:[m]\rAct\,[n]$ of the subsegments of $[m]$ are mutually disjoint, and cover the subsegments of $[n]$.\vspace{1ex}

In \cite{Se} Segal constructed a cardinality-preserving functor $\gamma_\Delta:\Delta\to\Gamma$ and based his infinite delooping machine on the existence of this functor. The following proposition shows that Segal's functor is actually a cardinality- and unit-preserving moment functor and up to isomorphism uniquely determined by this property.

\begin{prp}\label{augmentation}For each unital moment category of finite type $\,\CC$ there is an essentially unique cardinality- and unit- preserving moment functor $\,\gamma_\CC:\CC\to\Gamma$.\end{prp}

\begin{proof}Each object of $\CC$ has only finitely many elementary moments. This determines $\gamma_\CC$ on objects since the category $\Gamma$ has precisely one object for each finite cardinal. In particular, $\gamma_\CC$ is unit-preserving since units have cardinality $1$ by (U1).

In order to define the functor $\gamma_\CC$ on morphisms, we fix its image for each inert morphism $U\rIn A$ with unital domain in such a way that distinct elementary subobjects have distinct images in $\Gamma$. This amounts to fixing a bijection between the elementary moments of $A$ and the elementary moments of $\gamma_\CC(A)$, i.e. to a total ordering of the elementary moments of $A$. Each moment $\phi$ of $A$ determines then a subset of $\gamma_\CC(A)$, namely the one which corresponds to the elementary submoments of $\phi$. By Proposition \ref{inducedpartition}, pushforward along $f:A\to B$ takes distinct elementary moments of $A$ to disjoint moments of $B$. We therefore get a well-defined map $\gamma_\CC(f):\gamma_\CC(A)\to\gamma_\CC(B)$ in $\Gamma$, and this assignment is easily seen to be functorial. Lemma \ref{active=surj} shows that $\gamma_\CC$ takes active morphisms to active morphisms. Since pushforward along inert morphisms faithfully preserves elementary moments, $\gamma_\CC$ takes inert morphisms to inert morphisms so that $\gamma_\CC$ is indeed a moment functor.

Conversely, any functor of moment categories $\gamma_\CC:\CC\to\Gamma$ takes moments of $A$ to moments of $\gamma_\CC(A)$, and pushforward operation $f_*:m_A\to m_B$ to pushforward operation $\gamma_\CC(f)_*:m_{\gamma_\CC(A)}\to m_{\gamma_\CC(B)}$. Therefore, once a bijection between the elementary moments of $A$ and $\gamma_\CC(A)$ is fixed, there is no choice in defining $\gamma_\CC$. Different choices of bijections lead to canonically isomorphic augmentations.\end{proof}

\begin{dfn}\label{intrinsicwreath}Let $\CC,\DD$ be unital moment categories. The \emph{wreath product} $\CC\wr\DD$ is defined to be the category for which
\begin{itemize}\item objects are tuples $(A,B_\alpha)$ given by an object $A$ of $\,\CC$ and a family $(B_\alpha)_{\alpha\in\el_A}$ of objects of $\,\DD$ indexed by the set $\el_A$ of elementary moments of $A$;
\item morphisms are tuples $(f,f_\alpha^{\alpha'}):(A,B_\alpha)\to(A',B'_{\alpha'})$ given by a morphism $f:A\to A'$ and morphisms $f_\alpha^{\alpha'}:B_\alpha\to B_{\alpha'}$ whenever $\alpha'\leq f_*(\alpha)$ in $m_{A'}$.\end{itemize}\end{dfn}

\begin{prp}\label{unitalwreath}For unital moment categories $\CC,\DD$ the wreath product $\CC\wr\DD$ is a unital moment category. In particular, Joyal's $\Theta_n$ are unital moment categories.\end{prp}

\begin{proof}We first define an active/inert factorisation system in $\CC\wr\DD$ where a morphism $(f,f_\alpha^{\alpha'})$ is active (resp. inert) if $f$ is active (resp. inert) in $\CC$ and all $f_\alpha^{\alpha'}$ are active (resp. inert) in $\DD$. For a given $(f,f_\alpha^{\alpha'}):(A,B_\alpha)\to(A',B'_{\alpha'})$ we factor $f$ as $$A\overset{f_{act}}{\rAct}\tilde{A}\overset{f_{in}}{\rIn}A'$$and observe that for each $f_{\alpha}^{\alpha'}:B_\alpha\to B_{\alpha'}$ with $\alpha'\leq f_*(\alpha)$ there is a unique $\tilde{\alpha}\in\el_{\tilde{A}}$ such that $(f_{in})_*(\tilde{\alpha})=\alpha'$. Indeed, by Proposition \ref{inducedpartition} $(f_{act})_*$ induces a partition of $\el_{\tilde{A}}$ indexed by $\el_A$ so that each $\alpha'\leq (f_{in})_*(f_{act})_*(\alpha)$ is the pushforward of a uniquely determined elementary submoment $\tilde{\alpha}\leq(f_{act})_*(\alpha)$. Factor now $f_\alpha^{\alpha'}$ as$$B_\alpha\overset{(f_\alpha^{\alpha'})_{act}}{\rAct}\tilde{B}\overset{(f_\alpha^{\alpha'})_{in}}{\rIn}B'_{\alpha'}$$and index the middle object by $\tilde{\alpha}\in\el_{\tilde{A}}$. This defines an active/inert factorisation of $(f,f_\alpha^{\alpha'})$ which is essentially unique by the essential uniqueness of the active/inert factorisation systems in $\CC$ and $\DD$. Axioms (M1) and (M2) for the wreath product $\CC\wr\DD$ follow from their validity in $\CC$ and $\DD$.

By axiom (U1) units have a single elementary moment, the identity. Therefore, units of $\CC\wr\DD$ are of the form $(U,V)$ where $U$ (resp. $V$) is a unit of $\CC$ (resp. $\DD$). Axiom (U2) for $(U,V)$ is a consequence of its validity in $\CC$ (for $U$) and in $\DD$ (for $V$). Finally, each object $(A,B_\alpha)$ of $\CC\wr\DD$ receives an essentially unique active morphism $(U,V)\rAct(A,B_\alpha)$.

Joyal's categories $\Theta_n$ are iterated wreath products $\Delta\wr\cdots\wr\Delta$ according to \cite[Theorem 3.7]{Be2}. The wreath product $\Delta\wr-$ of \cite[Definition 3.1]{Be2} coincides with Definition \ref{intrinsicwreath} using the unital moment structure of $\Delta$, cf. Remark \ref{unitalexamples}b.\end{proof}

\begin{rmk}\label{internalwreath}The augmentation $\gamma_{\Theta_n}:\Theta_n\to\Gamma$ of Proposition \ref{augmentation} coincides up to isomorphism with the functor $\gamma_n:\Theta_n\to\Gamma$ of \cite[Definition 3.3]{Be2}.

The inductive definition of $\gamma_n$ involves a so-called \emph{assembly functor} $\alpha:\Gamma\wr\Gamma\to\Gamma$ which takes $(\nn;\mm_1,\dots,\mm_n)$ to $\underline{m_1+\cdots+m_n}$, cf. \cite[Lemma 3.2]{Be2} and Remark \ref{unitalexamples}a. This assembly functor induces \emph{wreath products} of automorphisms in $\Gamma$ entering into the formulation of the equivariance properties of a symmetric operad.

To be more precise, for $(\sigma;\tau_1,\dots,\tau_n)\in\Aut(\nn)\times\Aut(\mm_1)\times\cdots\times\Aut(\mm_n)$ we denote the image under $\alpha:\Gamma\wr\Gamma\to\Gamma$ by $\sigma\wr\tau_i\in\Aut(\underline{m_1+\cdots+m_n})$. Explicitly, this internal wreath product $\sigma\wr\tau_i$ is obtained by postcomposing the direct sum $\tau_1\oplus\cdots\oplus\tau_n$ with the obvious ``block permutation'' induced by $\sigma$.

Thanks to Proposition \ref{augmentation}, \emph{internal wreath products} can be defined for any unital moment category $\CC$. Let $f:A\rAct B$ be an active map in $\CC$ and choose splittings $B_\alpha\rIn B$ for the moments $f_*(\alpha)\in m_B$ obtained by pushing forward the elementary moments $\alpha\in\el_A$. For $(\sigma;\tau_\alpha)\in\Aut(A)\times\prod_{\alpha\in\el_A}\Aut(B_\alpha)$ we shall denote $\sigma\wr\tau_\alpha$ any automorphism of $B$ taken under $\gamma_\CC:\CC\to\Gamma$ to the wreath product $\gamma_\CC(\sigma)\wr\gamma_\CC(\tau_\alpha)$ in $\Gamma$. We shall not assume that these internal wreath products in $\CC$ exist nor that they are unique if they exist, but this will often be the case.\end{rmk}

\subsection{Operads}\label{operadsection}

In this section we shall introduce, for each unital moment category of finite type $\CC$, a notion of $\CC$-operad in an arbitrary symmetric monoidal category $\EE$. The unit of $\EE$ will be denoted $I_\EE$, and as usual, for a family $(E_i)_{i\in I}$ of objects in $\EE$ indexed by a finite set $I$, we shall denote $\bigotimes_{i\in I}E_i$ the tensor product of the $E_i$ where we do not choose any order on the tensor factors. This does not cause any harm because of Mac Lane's \emph{Coherence Theorem} for symmetric monoidal categories.

\begin{dfn}\label{operaddefinition}Let $\CC$ be a unital moment category of finite type.

We choose for each elementary moment $\alpha\in\el_A$, a splitting $U_\alpha\rIn A$, i.e. an \emph{elementary inert monomorphism}. The active/inert factorisation system induces then for each active morphism $f:A\rAct B$ inert morphisms $B_\alpha\rIn B$ splitting the moments $f_*(\alpha)\in m_B$ for $\alpha\in\el_A$. These splittings are normalised in the following sense: if $f_*(\alpha)=1_B$ then $B_\alpha=B$ and $B_\alpha\rIn B$ is the identity of $B$.

A \emph{$\CC$-collection} in $\EE$ consists of a functor $\Oo:\Iso(\CC)\to\EE$. This yields for each object $A$ of $\CC$ an $\Aut(A)$-object $\Oo(A)$ of $\,\EE$, and for each active morphism $f:A\rAct B$ in $\,\CC$, an object $\Oo(f)=\bigotimes_{\al\in\el_A}\Oo(B_\alpha)$ of $\,\EE$ endowed with a canonical $\prod_{\alpha\in\el_A}\Aut(B_\alpha)$-action.\vspace{1ex}

A \emph{$\CC$-operad} in $\,\EE$ is a $\CC$-collection in $\EE$ equipped with structure maps\begin{itemize}\item $\eta_U:I_\EE\to \Oo(U)$ (one for each unit $\,U$ of $\,\CC$)\item $\mu_f:\Oo(A)\otimes\Oo(f)\to\Oo(B)$ (one for each $f:A\rAct B$ in $\,\CC_{act}$)\end{itemize} such that the following unit, associativity and equivariance axioms hold:

\begin{itemize}\item[(i)]for each $f:U\rAct A$ with unital domain, the composite morphism$$I_\EE\otimes\Oo(A)\overset{\eta_U\otimes 1}{\longrightarrow} \Oo(U)\otimes \Oo(f)\overset{\mu_f}{\longrightarrow} \Oo(A)$$ is a left unit constraint in $\EE$;

\noindent for each identity $1:A\rAct A$, the composite morphism$$\Oo(A)\otimes \bigotimes_{\alpha\in\el_A}I_\EE\overset{1\otimes\bigotimes\eta_{U_\alpha}}{\longrightarrow} \Oo(A)\otimes\bigotimes_{\alpha\in\el_A} \Oo(U_\alpha)\overset{\mu_1}{\longrightarrow} \Oo(A)$$ is a right unit constraint in $\EE$;

\item[(ii)]for each pair $A\overset{f}{\rAct}B\overset{g}{\rAct}C$, the following diagram commutes:\begin{diagram}[small]\Oo(A)\otimes\Oo(f)\otimes\Oo(g)&\rTo^{1\otimes \mu_{f,g}}&\Oo(A)\otimes\Oo(gf)\\
\dTo^{\mu_f\otimes 1}&&\dTo_{\mu_{gf}}\\\Oo(B)\otimes\Oo(g)&\rTo_{\mu_g}&\Oo(C)\end{diagram}

\noindent where $\mu_{f,g}:\Oo(f)\otimes\Oo(g)\to\Oo(gf)$ is obtained (after permuting factors) as tensor product of the maps $\mu_{g_\alpha}:\Oo(B_\alpha)\otimes\Oo(g_\alpha)\to\Oo(C_\alpha),\,\alpha\in\el_A,$ induced by the following commutative diagram (cf. Proposition \ref{inducedpartition}):\begin{diagram}[small]A&\rAct^f&B&\rAct^g&C\\\uIn&&\uIn&&\uIn\\U_\alpha&\rAct_{f_\alpha}&B_\alpha&\rAct_{g_\alpha}&C_\alpha\end{diagram}

\item[(iii)]for each $f:A\rAct B$ and $(\sigma,\tau_\alpha)\in\Aut(A)\times\prod_{\alpha\in\el_A}\Aut(B_\alpha)$ admitting a wreath product $\sg\wr\tau_\alpha\in\Aut(B)$, cf. Remark \ref{internalwreath}, the following diagram \begin{diagram}[small]\Oo(A)\otimes\Oo(f)&\rTo^{\mu_f}&\Oo(B)\\\dTo^{\Oo(\sigma)\otimes\bigotimes_\alpha\Oo(\tau_\alpha)}&&
    \dTo_{\Oo(\sg\wr\tau_\alpha)}\\\Oo(A)\otimes\Oo(f)&\rTo_{\mu_{f\sigma^{-1}}}&\Oo(B)\end{diagram}commutes.\end{itemize}

The $\CC$-operads in $\EE$ form a category $\Oper_\CC(\EE)$ whose morphisms are maps of $\CC$-collections commuting with the structure maps $\eta_U$ and $\mu_f$. There is a restriction functor $\gamma_\CC^*:\Oper_\Gamma(\EE)\to\Oper_\CC(\EE)$. Its left adjoint $(\gamma_\CC)_!:\Oper_\CC(\EE)\to\Oper_\Gamma(\EE)$ is called \emph{symmetrisation}, cf. Remark \ref{operadexamples}c below.\end{dfn}

\begin{rmk}\label{operadexamples}Let us review the Examples \ref{examples}, see also Remark \ref{unitalexamples}:\vspace{1ex}

(a) A $\Gamma$-collection $\Oo$ is a collection of objects $\Oo_n$ endowed with $\Aut(\nn)=\Sigma_n$-actions. These are often called symmetric collections or symmetric sequences. A $\Gamma$-operad is then precisely a \emph{symmetric operad} in the sense of May \cite{May} where the multiplicative structure $\mu_f:\Oo_n\otimes\Oo_{m_1}\otimes\cdots\otimes\Oo_{m_n}\to\Oo_m$ is associated to the active map $f:\nn\rAct\mm$ with partition $\mm=\coprod_{\alpha\in\el(\nn)}f_*(\alpha)$, cf. Proposition \ref{inducedpartition}.

Note that for $\CC=\Gamma$, our axioms (i), (ii) and (iii) correspond precisely to May's axioms (b), (a) and (c) in \cite[Definition 1.1]{May}, except that we drop $\Oo(\underline{0})=I_\EE$.\vspace{1ex}

(b) A $\Delta$-operad is a \emph{non-symmetric operad}. Our symmetrisation functor is the usual functor assigning a symmetric operad to a non-symmetric operad.\vspace{1ex}

(c) A $\Theta_n$-collection is a collection of objects indexed by $n$-level trees. Note that there are no non-trivial automorphisms in $\Theta_n$. It turns out that a $\Theta_n$-operad is precisely a \emph{$(n-1)$-terminal $n$-operad} in the sense of Batanin \cite{Ba, Ba2}. This follows from the observation that Batanin's category of $n$-level trees $\Omega_n$ is the dual of $\Theta_n^{act}$ and that for $S\rAct T$ in $\Theta^{act}_n$, the induced partition of $T$ is indexed by the \emph{fibres} of the dual map in $\Omega_n$, cf. \cite[Definition 4.3]{Ba2}. Our symmetrisation functor coincides with Batanin's \cite[Section 8]{Ba2}. Its existence follows from Corollary \ref{extension} below.\end{rmk}

\subsection{Monoids}\label{monoidsection}

For each $\CC$-operad $\Oo$ in $\EE$ there is a natural notion of \emph{$\Oo$-algebra}. In this article, we will only be concerned with $I_\CC$-algebras where $I_\CC$ is the \emph{unit $\CC$-operad} defined by $I_\CC(A)=I_\EE$ for each object $A$ of $\CC$ with structure maps induced by the unit-constraints of $\EE$. Because of their importance $I_\CC$-algebras will simply be called \emph{$\CC$-monoids}. We now make explicit what $\CC$-monoids are without referring to the definition of $\Oo$-algebras for a general $\CC$-operad $\Oo$.

A \emph{$\CC$-monoid} $X$ assigns to each unit $U$ of $\CC$ an object $X(U)$ of $\EE$, and to each object $A$ the tensor product $X(A)=\bigotimes_{\alpha\in\el_A}X(U_\alpha)$ where the $U_\alpha$ are determined by the splittings $U_\alpha\rIn A$. The empty tensor product denotes the unit $I_\EE$ of $\EE$.

To each active morphism $f:A\rAct B$, the $\CC$-monoid assigns the tensor product $X(f)=\bigotimes_{\alpha\in\el(A)}X(f_\alpha)$ where $X(f_\alpha):X(B_\alpha)\to X(U_\alpha)$ is the value of $X$ at $f_\alpha:U_\alpha\rAct B_\alpha$, the latter being defined by the commutative square\begin{diagram}[small]A&\rAct^f&B\\\uIn&&\uIn\\U_\alpha&\rAct _{f_\alpha}&B_\alpha\end{diagram} These data must define a contravariant functor $X:\CC_{act}^\op\to\EE$, i.e. for each composable pair of active morphisms $A\rAct B\rAct C$ we must have $X(gf)=X(f)X(g)$ in $\EE$. In a similar manner as in Definition \ref{operaddefinition}ii, we identify $X(B)$ with the tensor product $\bigotimes_{\al\in\el_A}X(B_\alpha)$ which allows us to identify the target of $X(g)$ with $X(B)$. In other words, $\CC$-monoids are ``special'' presheaves on $\CC_{act}$, determined in a precise way by their values at units and at active morphisms with unital domain.

In particular $\Delta$-monoids (resp. $\Gamma$-monoids) in $\EE$ are precisely associative (resp. commutative) monoids in $\EE$ because $\Delta_{act}^\op$ (resp. $\Gamma_{act}^\op$) is the PRO (resp. PROP) for associative (resp. commutative) monoids, cf. Joyal \cite{J} and Mac Lane \cite{MacL}.

A $\Theta_2$-monoid has two multiplicative structures (induced by the two $2$-level trees of cardinality $2$) sharing the same unit and distributing over each other. By the Eckmann-Hilton argument $\Theta_2$-monoids are equivalent to commutative monoids, as are $\Theta_n$-monoids for any $n\geq 2$. Below, $\Theta_n$-monoids will simply be called \emph{$n$-monoids}.

There is a weakening of the notion of $\CC$-monoid in $\EE$ when $\EE$ is the category $\Sets^{\Delta^\op}$ of \emph{simplicial sets}. This weakening is based on the good behaviour of \emph{weak equivalences} (i.e. those simplicial maps whose topological realisation is a weak homotopy equivalence), including their stability under product and $2$-out-of-$3$, as well as the existence of a product preserving functor $\pi_0:\Sets^{\Delta^\op}\to\Sets$ taking weak equivalences to bijections. Moreover, it is crucial that the category of simplicial sets is \emph{cartesian closed}. Cartesian closedness implies that simplicial $\CC$-monoids arise from simplicial presheaves $\CC^\op\to\Sets^{\Delta^{\op}}$ provided the latter satisfy \emph{strict Segal conditions}. This leads to the following definition, see also Section \ref{strongunitality} below.

\begin{dfn}\label{Segalcondition}A \emph{$\CC_\infty$-monoid} for a unital (hyper)moment category $\CC$ is a simplicial presheaf $X:\CC^{\op}\to\Sets^{\Delta^\op}$ such that
\begin{itemize}\item[(i)]for nilobjects $A$, the value $X(A)$ is terminal;\item[(ii)]for non-nilobjects $A$ with representative set of elementary inert morphisms $A_\alpha\rIn A$, the map $s_A:X(A)\to\prod_{\al\in\el_A}X(A_\alpha)$ is a weak equivalence.\end{itemize}\end{dfn}

\noindent A simplicial presheaf on $\CC$ satisfying just (i) will be called \emph{reduced}.

If $\CC=\Delta$ or $\CC=\Gamma$ these structure maps $s_A$ are known as \emph{Segal maps} since they were introduced by Segal in \cite{Se}. We shall use the same terminology for \emph{any} unital (hyper)moment category $\CC$.

A $\CC_\infty$-monoid yields a simplicial $\CC$-monoid provided all Segal maps are invertible because in this case the naturality of the Segal maps induces for each active map $f:A\rAct B$ the components $X(f)=\prod_{\al\in\el(A)}X(f_\alpha)$ of a $\CC$-monoid. In particular, the set of path components of a $\CC_\infty$-monoid has the structure of $\CC$-monoid. For $\CC=\Delta,\Gamma,\Theta_n$ a $\CC_\infty$-monoid $X$ is called \emph{group-like} if $\pi_0(X)$ is a group.

\begin{exms}\label{table}The following table illustrates the definitions of this section. To the special cases studied so far are added the \emph{dendroidal category} $\Omega$ of Moerdijk-Weiss \cite{MW} and the \emph{graphoidal category} $\Gamma_\updownarrow$ of Hackney-Robertson-Yau \cite{HRY}, cf. Section \ref{ix} below, as well as the \emph{graphical category} $\UU$ for higher modular operads, introduced by the same authors \cite{HRY2}. These are unital \emph{hyper}moment categories (introduced in Section \ref{hyper} below) to which all definitions of this section apply (with a convenient adaptation of the group-like condition).

\begin{center}\begin{tabular}{|c|c|c|c|c|}
\hline $\CC$ & $\CC$-operad & $\CC$-monoid & $\CC_\infty$-monoid & group-like $\CC_\infty$-monoid\\
\hline $\Gamma$ & sym. operad & comm. monoid & $E_\infty$-space & infinite loop space\\
\hline $\Delta$& non-sym. operad & assoc. monoid & $A_\infty$-space & loop space\\
\hline $\Theta_n$&$n$-operad & $n$-monoid & $E_n$-space&$n$-fold loop space\\
\hline $\Omega$ & tree-hyperoperad & sym. operad & $\infty$-operad& (stable $\infty$-operad)\\
\hline $\Gamma_{\updownarrow}$ & directed hyperoperad & properad & $\infty$-properad& (stable $\infty$-properad)\\
\hline $\UU$ & hyperoperad & modular operad & $\infty$-modular op. & (stable $\infty$-modular op.)\\
\hline\end{tabular}\end{center}\vspace{1ex}

The last two columns should be interpreted as follows: there are two \emph{Quillen model structures} on the category of reduced simplicial presheaves on $\CC$ such that the fibrant objects are respectively $\CC_\infty$-monoids and grouplike $\CC_\infty$-monoids, and the homotopy category is equivalent to the homotopy category of the claimed objects.

For $\Delta,\Gamma$ these results go back to Segal \cite{Se} and Bousfield-Friedlander \cite{BF}.
For $\Theta_n$ the $E_n$-model structure may be obtained by restricting Rezk's model structure for weak $n$-categories \cite{R} to reduced simplicial presheaves, while the model structure for $n$-fold loop spaces is described in \cite{Be2}.

The last three rows are more involved because $\CC$-monoids have a more complicated structure here. The Segal model structure for $\infty$-operads has been described by Cisinski-Moerdijk \cite{CM} (restricting it to reduced simplicial presheaves). Related model categories have been investigated by Barwick \cite{Ba} and Chu-Haugseng-Heuts \cite{CHH}. The model structure for $\infty$-properads can be obtained by restricting the model structure of Hackney-Robertson-Yau \cite{HRY1}. The last row follows from recent work of the same authors \cite{HRY2}, see also \cite[Example 6.2]{H}. The resulting notion of $\UU$-operad is close to what Getzler and Kapranov call a \emph{hyperoperad} in \cite[Section 4]{GK} disregarding any genus labeling. The role played in \cite{HRY2} by the nilobjects (free-living edge and nodeless loop) in $\UU$ differs slightly from our reducedness condition.

It would of course be desirable to have a uniform proof of existence for these model structures. The key idea is to homotopically invert the Segal maps starting from a projective (or injective) model structure on the category of reduced simplicial presheaves on $\CC$. We hope to come back to this topic in a future paper.\end{exms}

\section{Hypermoment categories}\label{hyper}

This section introduces \emph{hypermoment} categories, which are more general than moment categories. A hypermoment category comes equipped with an active/inert factorisation system where there is no correspondence between inert subobjects and moments, but just an augmentation $\gamma_\CC:\CC\to\Gamma$ compatible with the active/inert factorisation systems. This induces a well-behaved notion of \emph{cardinality} for the objects of $\CC$ and is enough to define $\CC$-operads and $\CC$-monoids like in Section \ref{operadsection}. We then define a \emph{plus construction} $\CC^+$ for unital hypermoment categories $\CC$ with the characteristic property that $\CC$-operads get identified with $\CC^+$-monoids.

Like in the original plus construction of Baez-Dolan \cite{BD} this is achieved by taking ``basic operators'' (i.e. active morphisms with unital domain) in $\CC$ to ``types'' (i.e. units) in $\CC^+$, and ``reduction laws'' in $\CC$ to ``operators'' in $\CC^+$.

It turns out that $\CC^+$ can be constructed as a \emph{category of special elements} of the simplicial nerve of $\CC$. A closely related construction for operator categories has been considered by Barwick \cite{Ba} and further studied by Chu-Haugseng-Heuts \cite{CHH}. Our plus construction $\CC^+$ is in general different from theirs because it also depends on the inert part of $\CC$ inexistant in an operator category.

The dendroidal category of Moerdijk-Weiss \cite{MW} and the graphoidal category of Hackney-Robertson-Yau \cite{HRY} are examples of hypermoment categories which are not moment categories because both categories have inert morphisms without active retraction. Nonetheless both are augmented over $\Gamma$ by assigning to a dendrix (resp. graphix) its \emph{vertex set}. We will show that $\Omega$ contains the plus construction $\Gamma^+$ of Segal's category $\Gamma$ as a subcategory. This is the reason for which in Table \ref{table} symmetric operads appear twice: as $\Gamma$-operads and as $\Omega$-monoids.

We finally deduce from the existence of the plus construction $\CC^+$ a general monadicity result for $\CC$-operads, viewed as structures on $\CC$-collections. A crucial intermediate step is the construction of the free $\CC^+$-monoid generated by a $\CC$-collection. Here we need that $\CC$ is \emph{strongly extensional} and thus admits coherent pushouts of inert morphisms along active morphisms. These pushouts are used to define an abstract insertion of $\CC$-trees into vertices of $\CC$-trees.

\begin{dfn}\label{hypermoment}A \emph{hypermoment category} is a category $\CC$ equipped with an active/inert factorisation system and an augmentation $\gamma_\CC:\CC\to\Gamma$ such that\begin{itemize}\item[(i)]$\gamma_\CC$ preserves active (resp. inert) morphisms;\item[(ii)]$\gamma_\CC$ preserves cardinality in the following sense: for each object $A$ of $\,\CC$ and each element $\11\rIn\gamma_\CC(A)$, there is an essentially unique inert lift $U\rIn A$ in $\CC$ such that $U$ satisfies unit-axiom (U2) of Definition \ref{unitaldefinition}.\end{itemize}

A hypermoment category is called \emph{unital} if every object $A$ of $\CC$ receives an essentially unique active morphism $U\rAct A$ whose domain $U$ is a \emph{unit} of $\,\CC$, i.e. belongs to $\gamma_\CC^{-1}(\11)$ and satisfies unit-axiom (U2) of Definition \ref{unitaldefinition}.\end{dfn}

According to Proposition \ref{augmentation} each unital moment category $\CC$ of finite type admits an essentially unique augmentation $\gamma_\CC:\CC\to\Gamma$ turning it into a unital hypermoment category. The essential difference between the two notions is that in a hypermoment category there might exist inert morphisms \emph{without} active retraction. This implies that for a given object $A$ of a unital hypermoment category $\CC$ the domain of the (essentially unique) active morphism $U\rAct A$ might be different from the domains of inert morphisms $U_\alpha\rIn A$. This possibility is excluded in a unital moment category because composing $U\rAct A$ with the active retraction $A\rAct U_\alpha$ yields an active morphism $U\rAct U_\alpha$ between units which is necessarily invertible by Lemma \ref{incomparable}i. For instance, we will see below that for a dendrix $A$, the active map $U\rAct A$ encodes the leaves of the dendrix while an inert map $U_\alpha\rIn A$ encodes the edges incoming into vertex $\alpha$ so that $U\not\cong U_\alpha$ in general.

Despite of this extra-freedom available in a unital hypermoment category, the definitions of $\CC$-operad and $\CC$-monoid (cf. Sections \ref{operadsection} and \ref{monoidsection}) can be copied almost verbatim, provided we use instead of elementary moments \emph{elementary inert subobjects}. Indeed, the crucial pushforward operations of Section \ref{pushforward} apply to inert subobjects without any modification. In particular, from now on, \emph{we shall denote by $\el_A$ the set of elementary inert subobjects of $A$}, i.e. the set of isomorphism classes of inert morphisms with unital domain and fixed codomain $A$.

Before turning to examples we establish an important relationship between \emph{rigid} hypermoment categories and \emph{operadic} categories in the sense of Batanin-Markl \cite{BaM}.

A unital hypermoment category $\CC$ is said to be \emph{rigid} if every isomorphism is an automorphism, and every automorphism acts trivially (on the left) on active morphisms with unital domain. In particular, units are ``rigid objects'' insofar as they do not have any non-trivial automorphisms. For instance, the moment categories $\Gamma,\Delta$ and $\Theta_n$ of Examples \ref{examples} are rigid, but the hypermoment categories $\Omega$ and $\Gamma_\updownarrow$ treated in Section \ref{ix} below are not. Nonetheless, we shall see in Appendix \ref{rigidify} that there is a combinatorial way of rigidifying $\Omega$ and $\Gamma_\updownarrow$ so that the following proposition can be applied to their rigidifications.

\begin{prp}\label{operadic}The dual of the active part of a rigid hypermoment category has a canonical structure of operadic category in the sense of Batanin-Markl \cite{BaM}.\end{prp}

\begin{proof}See Appendix \ref{proofoperadic}.\end{proof}

Roughly speaking, the pushforward operations of a rigid hypermoment category $\CC$ give rise to cofibre functors on $\CC_{act}$ which are dual to fibre functors on $\CC_{act}^\op$ as defined by Batanin-Markl \cite{BaM}. This restriction/dualisation process tends to render the categories $\CC_{act}^\op$ closer to \emph{set-theoretical} intuition because they are augmented over $\Gamma_{act}^\op$, a skeleton of the category of finite sets and maps between them. It is an interesting yet difficult problem to determine which operadic categories arise through this restriction/dualisation process from rigid hypermoment categories.

\begin{rmk}\label{symmetries}Our notion of $\CC$-operad for a rigid hypermoment category $\CC$ is almost the same as Batanin-Markl's notion of operad over the corresponding operadic category $\CC_{act}^\op$. The difference concerns equivariance which is missing in \cite{BaM}. Nonetheless, an operad over an operadic category may have symmetries induced by the operad multiplication. Our equivariance axiom amounts to requiring that these ``external'' symmetries coming from the operad multiplication coincide with the ``internal'' symmetries coming from the automorphisms of $\CC$. So, the principal difference is that in our approach $\CC$-operads are viewed as structures on collections with symmetries while in the approach of Batanin-Markl they are viewed as structures on collections without symmetries. Batanin shows in \cite[Proposition 3.1]{Ba2} that for symmetric operads both view points are equivalent.\end{rmk}

\subsection{Dendrices and graphices}\label{ix}

The \emph{dendroidal category} $\Omega$ has been introduced by Moerdijk-Weiss \cite{MW}, see also \cite{HM} for a recent presentation.

The objects of $\Omega$ (the \emph{dendrices}) are finite rooted trees, the morphisms of $\Omega$ are defined by viewing such trees as coloured symmetric operads, where the colours are the edges of the tree, and the operations are freely generated by the vertices of the tree. The morphisms of $\Omega$ are thus maps of coloured symmetric operads. The dendrices may have vertices with a single incident edge (i.e. \emph{stumps}) representing constant operations of the induced coloured symmetric operad. We refer to Appendix \ref{dendrixappendix} (especially Sections \ref{Gammavsdendrixinert} and \ref{Gammavsdendrixactive}) for a more rigorous definition.

The simplex category $\Delta$ may be identified with the full subcategory of $\Omega$ spanned by the linear trees without stumps. The moment structure of $\Delta$ extends to a hypermoment structure of $\Omega$ where $\gamma_\Omega:\Omega\to\Gamma$ takes a dendrix to its vertex set.

Mixing the active/inert factorisation system with the epi/mono factorisation system (cf. Kock \cite{K0}, \cite[Section 2.15]{K2}) induces a triple factorisation system for $\Omega$: each morphism can be written in an essentially unique way as a degeneracy operator followed by an inner face operator followed by an outer face operator.

Degeneracy operators (i.e. retractive morphisms) correspond to dropping vertices with exactly two incident edges. Outer face operators (i.e. \emph{inert} morphisms)  $S\rIn T$ can be viewed as \emph{embeddings}. Inner face operators (i.e. active monomorphisms) $S\rAct T$ are dual to inner edge contractions of $T$ and can be viewed as partitions of $T$ into subdendrices $T_\alpha\rIn T$ indexed by the vertex set of $S$. Alternatively, an inner face operator $S\rAct T$ can also be viewed as an \emph{insertion} of dendrices $T_\alpha$ into the vertices $\alpha$ of $S$ such that $T$ is the result of this insertion process. The \emph{active} part of $\Omega$ is generated by degeneracies and inner face operators.

In contrast to $\Delta$, the dendroidal category $\Omega$ has non-trivial symmetries, the dendrix automorphisms. Also, in contrast to $\Delta$, there are inert morphisms without active retraction, namely those $S\rIn T$ for which the complement of $S$ in $T$ does \emph{not} decompose into a coproduct of linear trees.

The dendroidal category has a single nilobject: the edge $|$ without vertices. The \emph{units} of $\Omega$ are precisely the corollas $C_n$ where $n$ is the number of leaves. Unit-axiom (U2) is satisfied because the source of a degeneracy $S\rAct C_n$ has a unique vertex mapping to the vertex of $C_n$ so that there is a uniquely determined inert section $C_n\rIn S$. Moreover, this property characterises corollas. Axiom (ii) of Definition \ref{hypermoment} says then that for each vertex $\alpha$ of a dendrix $T$ there is an essentially unique pair consisting of a corolla $C_{n(\alpha)}$ and an outer face operator $C_{n(\alpha)}\rIn T$, where $n(\alpha)$ is the valency of $\alpha$. The dendroidal category is \emph{unital} because for each dendrix $T$ there is also an essentially unique pair consisting of a corolla $C_{n(T)}$ and an inner face operator $C_{n(T)}\rAct T$ where $n(T)$ is the number of leaves of $T$.

It can now be checked that $\Omega$-operads in the sense of Definition \ref{operaddefinition} are \emph{tree hyperoperads} in the spirit of Getzler-Kapranov's hyperoperads \cite{GK}. This analogy motivated our terminology of hypermoment category. It can also be checked by hand that $\Omega$-monoids are (single-coloured) symmetric operads. An early account of this last equivalence can be found in Ginzburg-Kapranov \cite[Section 1.2]{GiK}.

Hackney, Robertson and Yau \cite{HRY} further embed the dendroidal category $\Omega$ into a \emph{graphoidal category} $\Gamma_{\updownarrow}$. Its objects (the \emph{graphices}) are finite connected graphs with \emph{directed} edges and \emph{directed} leaves so that there are \emph{no oriented edge-cycles} in the graph. Each dendrix defines a graphix by directing all edges towards the root.

The morphisms of $\Gamma_{\updownarrow}$ are definable like the morphisms of $\Omega$, cf. Chu-Hackney \cite[Section 2.2]{CHa} and \cite{HRY}. It is best to describe directly the triple factorisation system for $\Gamma_{\updownarrow}$. Degeneracies correspond to dropping vertices with exactly one incoming and one outgoing edge. Outer face operators (i.e. inert morphisms) are graphix embeddings, while inner face operators (i.e. active monomorphisms) correspond to insertion of graphices into vertices of graphices. The augmentation $\Gamma_{\updownarrow}\to\Gamma$ takes a graphix to its vertex set.

There is a single nilobject, the directed edge $\updownarrow$ without vertices. The units are directed corollas $C_{m,n}$ with $m$ incoming leaves and $n$ outgoing leaves. The graphoidal category is a unital hypermoment category because for each graphix $G$ there are essentially unique inert morphisms $C_{m(\alpha),n(\alpha)}\rIn G$, resp. active morphism $C_{m(G),n(G)}\rAct G$ determined by the vertices $\alpha$, resp. the leaves of $G$.

$\Gamma_{\updownarrow}$-operads are directed graph hyperoperads in the spirit of Getzler-Kapranov \cite{GK} while $\Gamma_{\updownarrow}$-monoids are (set-based) \emph{properads} in the sense of Vallette \cite{V}. Indeed, the underlying object of a $\Gamma_\updownarrow$-monoid is a presheaf on $(\Gamma_\updownarrow)_{unit}$, i.e. a collection of $\Sg_m\times\Sg_n$-objects where $\Sg_m\times\Sg_n$ is the automorphism group of a corolla $C_{m,n}$ with $m$ incoming and $n$ outgoing leaves ($m\geq 0,n\geq 0$). A properad structure \cite{V} on such a bisymmetric collection amounts precisely to a $\Gamma_\updownarrow$-monoid structure in the sense of Section \ref{monoidsection}. The reader should note that the morphism-set $\Gamma_{\updownarrow}(G,H)$ is contained in, but in general not equal to, the set of properad morphisms between the free properads on the graphices $G$ and $H$, cf. \cite[Sections 5-6]{HRY}. This phenomenon is well explained by Kock \cite[Section 2.4.14]{K}, cf. also Chu-Hackney \cite[Theorem 2.2.19]{CHa}: the set of all such properad maps only admits a \emph{weak} active/inert factorisation system for which the inert part also contains certain free maps which are not monomorphic. The subset $\Gamma_\updownarrow(G,H)$ only retains those properad maps whose inert part is monomorphic, thereby producing a genuin (orthogonal) active/inert factorisation system on $\Gamma_\updownarrow$.

\subsection{Plus construction}

We introduce an analog of the plus construction of Baez-Dolan \cite[Definition 15]{BD} for unital hypermoment categories $\CC$. Its characteristic property is that $\CC$-operads get identified with $\CC^+$-monoids.

The original Baez-Dolan construction was conceived for coloured symmetric operads with a similar universal property in mind. In literature, plus constructions have been proposed for \emph{polynomial monads} (cf. \cite[Section 11]{BB}, \cite[Section 2.2]{K2}), for \emph{Feynman categories} (cf. \cite[Section 3.6]{KW}), for \emph{operadic categories} (cf. \cite[Section 5]{BaM2}), and for \emph{symmetric monoidal categories} (cf. \cite[Section 3]{KM}). Each context has its own specificities and it is not obvious how to switch from one context to the other, let alone the switch from one plus construction to the other.

Our plus construction takes care of symmetries and has an elementary categorical definition. We obtain a non-full embedding of the plus construction $\Gamma^+$ of Segal's category $\Gamma$ into the dendroidal category $\Omega$ of Moerdijk-Weiss \cite{MW}. This witnesses the importance of tree combinatorics involved in the passage from $\CC$-operads to $\CC^+$-monoids. Closely related constructions may be found in \cite{HHM,B,CHH,BaM2}.

\begin{dfn}\label{hyperdefinition}Let $\CC$ be a unital hypermoment category.
\begin{itemize}\item A $\,\CC$-tree is a pair $([m],A_0\rAct\cdots\rAct A_m)$ consisting of an object $[m]$ of $\Delta$ and a functor $A_\bullet:[m]\to\CC_{act}$ such that $A_0$ is a unit of $\,\CC$ and no object among the $A_i$ is a nilobject of $\,\CC$ except possibly $A_m$;\item A $\,\CC$-tree morphism is pair $(\phi,f)$ consisting of a morphism $\phi:[m]\to[n]$ in $\Delta$ and a natural transformation $f:A\to B\phi$ which is pointwise inert, i.e. $f_i:A_i\to B_{\phi(i)}$ is inert in $\CC$ for all $i\in[m]$;\end{itemize}

A $\CC$-tree morphism $(\phi,f)$ is called \emph{active} (resp. \emph{inert}) if $\phi$ is active and $f$ invertible (resp. if $\phi$ is inert).

A \emph{vertex} of $\,([m],A_\bullet)$ is an elementary inert subobject $U\rIn A_i$ for some $i<m$. Vertices will be represented by inert morphims $([1],U\rAct A)\rIn([m],A_\bullet)$ where $U\rAct A\rIn A_{i+1}$ is the active/inert factorisation of $U\rIn A_i\rAct A_{i+1}$.

The plus construction $\CC^+$ is the category of $\,\CC$-trees and $\,\CC$-tree morphisms.\end{dfn}

\begin{prp}\label{plus}The plus construction takes unital hypermoment categories to unital hypermoment categories. The augmentation $\gamma_{\CC^+}:\CC^+\to\Gamma$ takes a $\CC$-tree to its vertex set. Units of~$\,\CC^+$ are $\CC$-trees $([1],U\rAct A)$ where $U$ is a unit of $\,\CC$.\end{prp}

\begin{proof}For each $\CC$-tree $([m],A_\bullet)$ and active map $\phi:[m]\rAct\,[m']$ such that $A_i\cong A_j$ whenever $\phi(i)=\phi(j)$, let us construct an active $\CC$-tree morphism $(\phi,f):([m],A_\bullet)\rAct([m'],A'_\bullet)$. Indeed, we put $A'_i=A_j$ where $j\in[m]$ is the least integer such that $i\in[\phi(j),\phi(j+1)]$. The morphisms $A'_0\rAct\cdots\rAct A'_{m'}$ are then defined by composing the obvious morphisms in $A_\bullet$, and we get in this way an active $\CC$-tree morphism $(\phi,f)$ because of the compatibility between $A_\bullet$ and $\phi$.

For a $\CC$-tree morphism $(\phi,h):([m],A_\bullet)\to([n],B_\bullet)$ write $\phi=\phi_{in}\phi_{act}$. The definition of $\CC$-tree morphism implies that $A_i\cong A_j$ whenever $\phi_{act}(i)=\phi_{act}(j)$. We can thus construct an active $\CC$-tree morphism $(\phi_{act},f):([m],A_\bullet)\rAct([m'],A'_\bullet)$ followed by an inert $\CC$-tree morphism $(\phi_{in},g):([m'],A'_\bullet)\rIn([n],B_\bullet)$ such that $(\phi,h)=(\phi_{in},g)(\phi_{act},f)$. This factorisation is essentially unique because it is unique on the first factor and essentially unique on the second factor.

The augmentation $\gamma_{\CC^+}:\CC^+\to\Gamma$ is defined on objects by sending a $\CC$-tree to the set of its vertices. We shall turn this assignment into a functor by defining it as a composite of three moment functors $\CC^+\to\Delta\wr\CC\to\Gamma\wr\Gamma\to\Gamma$ where the last functor is the assembly functor (cf. Remark \ref{internalwreath}) and the middle functor is the wreath product $\gamma_\Delta\wr\gamma_\CC$ of the two augmentations. It thus remains to define the first functor. On objects it takes $([m],A_\bullet)$ to $([m];A_0,A_2,\dots,A_{m-1})$, cf. Definition \ref{intrinsicwreath}. For a $\CC$-tree morphism $(\phi,f):([m],A_\bullet)\to([n],B_\bullet)$ we have to define morphims $A_i\to B_j$ in $\CC$ whenever $[j,j+1]\subset[\phi(i),\phi(i+1)]$. By hypothesis we have a morphism $A_i\rIn B_{\phi(i)}$ from which we obtain the required morphism by composition with $B_{\phi(i)}\rAct\cdots\rAct B_j$. The naturality of $A\to B\phi$ ensures that this defines indeed a functor $\CC^+\to\Delta\wr\CC$. If $(\phi,f)$ is an active (resp. inert) $\CC$-tree morphism then its image in $\Delta\wr\CC$ is an active (resp. inert) morphism.

Let us determine the units of $\CC^+$. By definition, these are the objects of $\CC^+$ with a unique vertex and subject to unit-axiom (U2). In particular, they must be of the form $([1],A_\bullet)$ because otherwise they would have more than one vertex (if $m>1$) or none (if $m=0$). Unit axiom (U2) requires any active
$\CC$-tree morphism with target a unit $([1],A_\bullet)$ to admit a unique inert section. This holds actually for any $\CC$-tree $([1],A_\bullet)$ because it holds on the first factor and, in general, an active $\CC$-tree morphism $(\phi,f)$ is retractive if and only if $\phi$ is. The augmentation $\gamma_{\CC^+}:\CC^+\to\Gamma$ satisfies the lifting condition (ii) of Definition \ref{hyperdefinition} by construction.

For an arbitrary $\CC$-tree $([m],A_\bullet)$ the total composition $A_0\rAct A_m$ yields a unit $([1],A_0\rAct A_m)$ of $\CC^+$. If $\CC$ is unital, the $\CC$-tree $([m],A_\bullet)$ receives an essentially unique active $\CC$-tree map from $([1],A_0\rAct A_m)$ so that $\CC^+$ is unital as well.\end{proof}

\begin{thm}\label{BaezDolanplus}For each unital hypermoment category $\,\CC$, the categories of $\,\CC$-operads and of $\,\CC^+$-monoids are equivalent.\end{thm}

\begin{proof}We first show that the data underlying a $\CC$-operad and a $\CC^+$-monoid are equivalent. It follows from Section \ref{monoidsection} and Proposition \ref{plus} that a $\CC^+$-monoid is a functor $X:(\CC^+_{act})^\op\to\EE$ such that for each $\CC$-tree we have $X([m],A)=\bigotimes_{\alpha\in\gamma_{\CC^*}([m],A_\bullet)}X([1],U_\alpha\rAct A_\alpha)$ where $([1],U_\alpha\rAct A_\alpha)\rIn([m],A_\bullet)$ is a vertex. Since $\CC$ is unital, the functor which associates to a unit $([1],U_\alpha\rAct A_\alpha)$ of $\CC^+$ the object $A_\alpha$ of $\CC$ induces an equivalence of categories $(\CC^+)_{unit}\overset{\sim}{\to}\CC_{iso}$ so that the underlying objects of a $\CC^+$-monoid and a $\CC$-operad are indeed equivalent. Let us denote $\Oo_X:\CC_{iso}\to\EE$ the corresponding underlying object of a $\CC$-operad.

For each unit $U$ of $\CC$, the object $\Oo_X(U)$ of $\EE$ has a distinguished ``element'' because $X([1],1_U)$ has a distinguished ``element'' $I_\EE=X([0],U)\to X([1],1_U)$ induced by the uniquely determined morphism $([1],1_U)\rAct([0],U)$ in $\CC^+_{act}$.

The multiplication $\mu_f:\Oo_X(A)\otimes\Oo_X(f)\to\Oo_X(B)$ for an active map $f:A\rAct B$ in $\CC$ corresponds to the $X$-action of the morphism $$([1], U\rAct B)\rAct([2],U\rAct A\overset{f}{\rAct} B)$$ in $\CC^+_{act}$. Unit-, associativity- and equivariance axioms of an $\Oo$-operad $\Oo_X$ are then codified by the $X$-action of certain commutative diagrams in $\CC^+_{act}$. For instance, the associativity constraint is induced by the commutativity of the following diagram$$\xymatrix{([3],U\rAct A\overset{f}{\rAct} B\overset{g}{\rAct} C)&([2],U\rAct A\overset{gf}{\rAct} C)\ar[l]\\
([2],U\rAct B\overset{g}{\rAct} C)\ar[u]&([1],U\rAct C)\ar[u]\ar[l]}$$in $\CC^+_{act}$. A consistent $X$-action with respect to these commutative diagrams is enough to define a $\CC^+$-monoid $X:(\CC^+_{act})^\op\to\EE$ since simplicial nerves are determined by their $3$-skeleton. Under this correspondence between $\CC$-operads and $\CC^+$-monoids the respective morphisms correspond as well.\end{proof}

There are variants of the dendroidal category $\Omega$ compatible with the active/inert factorisation system. A dendrix is called \emph{open} (resp. \emph{closed}) if it has no stumps (resp. no leaves).  The active/inert factorisation restricts to the full subcategory $\Omega_o$ (resp. $\Omega_c$) of open (resp. closed) dendrices. Since the corolla without leaves is closed while all other corollas are open, neither $\Omega_o$ nor $\Omega_c$ contain all units of $\Omega$.

We call a dendrix \emph{reduced} if it is either a closed dendrix or has the property that its leaves are precisely the edges of maximal height. In particular, every corolla is a reduced dendrix. A morphism of dendrices is said to be \emph{reduced} if its active/inert factorisation factors through a reduced dendrix and the morphism takes edges of same height to edges of same height. Reduced morphisms compose and the subcategory $\Omega_r$ consisting of reduced dendrices and reduced dendrix morphisms is a hypermoment subcategory of $\Omega$ containing all units of $\Omega$.

A dendrix is \emph{planar} if for each vertex the set of incoming edges is equipped with a linear ordering. The category $\Omega^{pl}$ of planar dendrices and planar morphisms is an example of a rigid hypermoment category. The subcategory $\Omega_r^{pl}$ consisting of reduced planar dendrices and reduced planar morphisms is a hypermoment subcategory of $\Omega^{pl}$ containing all units of $\Omega^{pl}$.

\begin{prp}\label{Gammaplus}The plus construction $\Gamma^+$ (resp. $\Delta^+$) is equivalent to the hypermoment category $\Omega_r$ (resp. $\Omega_r^{pl}$) of reduced (resp. reduced planar) dendrices.\end{prp}

\begin{proof}See Appendix \ref{dendrixappendixproof}.\end{proof}

\begin{rmk}\label{CHH}Chu-Haugseng-Heuts \cite{CHH} consider a subcategory $\Delta^1_\FF$ of the category of forests $\Delta_\FF$ introduced by Barwick \cite{B}. In the following discussion we consider an even smaller category $\tilde{\Delta}^1_\FF$, namely the full subcategory of $\Delta^1_\FF$ whose simplices of objects of $\FF$ contain at most once the emptyset, in accordance with our convention concerning $\Gamma$-trees in Definition \ref{hyperdefinition}. With this convention, we get an isomorphism of categories $\Gamma^+\cong\tilde{\Delta}^1_\FF$. Indeed, the dual of $\Gamma_{act}$ is isomorphic to a skeleton $\FF$ of the category of finite sets, and naturality squares in $\Gamma$ (left)$$\xymatrix{\ar@{ >->}[r]&&&\ar[l]|{+}&\ar[d]\ar@{^{(}->}[r]&\ar[d]\\\ar[u]|{+}\ar@{ >->}[r]&\ar[u]|{+}&\ar[u]|{+}&\ar[l]|{+}\ar[u]|{+}&\ar@{^{(}->}[r]&}$$are equivalent by Lemma \ref{retractive} and Axiom (MC) to pushout squares in $\Gamma_{act}$ with retractive horizontal maps (middle), which in turn are equivalent to pullback squares in $\FF$ with horizontal inclusions (right). This shows that the naturality squares used in defining the morphisms of $\Gamma^+$ and of $\tilde{\Delta}^1_\FF$ correspond to each other and have the same horizontal variance. The objects of $\Gamma^+$ and $\tilde{\Delta}^1_\FF$ are also in canonical one-to-one correspondence (reversing the orientation of the vertical arrows).

In \cite[Section 4]{CHH} a functor $\Delta^1_\FF\to\Omega$ is constructed which upon inspection (using our convention above and Kock's description \cite{K0} of $\Omega$ by means of tree-polynomials) induces the equivalence of categories $\Gamma^+\cong\tilde{\Delta}^1_\FF\simeq\Omega_r$ of Proposition \ref{Gammaplus}.

By \cite[Theorem 5.1]{CHH} the resulting hypermoment category inclusion $\Omega_{r}\inc\Omega$ induces an equivalence of Segal type homotopy theories for simplicial presheaves on both sides. This is related to work of Heuts-Hinich-Moerdijk \cite{HHM} where another category of forests than Barwick's $\Delta_\FF$ \cite{B} is used to compare the Cisinski-Moerdijk model \cite{CM} with the Lurie model \cite{Lu} for $\infty$-operads, see also \cite[Lemma 2.11]{CHH}.

The planar version of Proposition \ref{Gammaplus} is related to Baez-Dolan's \emph{$n$-opetopes} \cite{BD}. Starting with the simplex category $\Delta$ we can apply $n$ times our plus construction:  the \emph{units} of the resulting hypermoment category can then be viewed as a special kind of $n$-opetopes. Indeed, for any hypermoment category over $\Delta$, the plus construction comes equipped with a functor $\CC^+\to\Delta^+$. In particular, $\CC$-trees have an underlying $\Delta$-tree which is a reduced planar dendrix. This can be viewed as an interpretation of the slogan that opetopes are ``trees of trees of tress of ...''. An interesting feature of our approach is the presence of inert morphisms and of degeneracies which might reveal so far hidden aspects of opetopes.\end{rmk}

\subsection{Segal cores, strong unitality and extensionality}\label{Segalcore}

We discuss here two properties which are present in all hypermoment categories so far discussed. \emph{Strong unitality} permits a reformulation of the Segal conditions, cf. Definition \ref{Segalcondition}. \emph{Extensionality} allows us to define insertion of $\CC$-trees into vertices of $\CC$-trees. We also introduce the notion of \emph{Segal core} of a unital hypermoment category which is closely related to the definition of an ``algebraic pattern'' by Chu-Haugseng \cite{CH}.

For the notion of dense subcategory, see e.g. \cite{BMW}, especially Lemma 1.7 therein.

\begin{dfn}\label{strongunitality}
The \emph{Segal core} $\CC_\Seg$ of a hypermoment category $\CC$ is the full subcategory of the inert part spanned by the units and the nilobjects.

A unital hypermoment category is called \emph{strongly unital} if its Segal core $\,\CC_\Seg$ is \emph{dense} in the inert part $\CC_{in}$.\end{dfn}

This means that each object of $\CC$, when viewed as an object of the inert part, is a \emph{canonical} colimit of unit- and nilobjects. A simplicial presheaf $X:\CC^\op\to\Sets^{\Delta^\op}$ is then said to be a \emph{strict Segal presheaf} (resp. \emph{Segal presheaf}) if its restriction to the inert part takes the density colimit cocones to limit cones (resp. homotopy limit cones) in simplicial sets. The advantage of this refined Segal condition is that it applies to general simplicial presheaves on $\CC$. If the latter are reduced (cf. Definition \ref{Segalcondition}i) then the (homotopy) limit cones are actually (homotopy) product cones, and we recover the Segal condition of Definition \ref{Segalcondition}ii. We shall denote the category of set-valued Segal presheaves on $\CC$ by $\Pp_\Seg(\CC)$.

Let us indicate the Segal cores of our main examples. Let us also mention the concomitant notion of \emph{$\CC$-graph}. A \emph{$\CC$-graph} is a set-valued presheaf on the Segal core $\CC_\Seg$. The category of $\CC$-graphs will be denoted $\Pp(\CC_\Seg)$.\vspace{1ex}

\begin{center}\begin{tabular}{|c|c|c|c|c|c|}
\hline $\CC$ & $\Gamma$ & $\Delta$ & $\Theta_n$ & $\Omega$ & $\Gamma_\updownarrow$\\
\hline $\CC_\Seg$& $\underline{0}\to\underline{1}$ & $[0]\rightrightarrows[1]$ &cell-incl. of&edge-incl. of&edge-incl. of\\
&&&glob. $n$-cell&corollas&dir. corollas\\
\hline $\Pp(\CC_\Seg)$ & graded object& graph& $n$-graph&coloured coll.&dir. col. coll.\\
\hline $\Pp_\Seg(\CC)$ & gr.com. monoid&category&$n$-category&coloured operad&col. properad\\

\hline\end{tabular}\end{center}\vspace{1ex}

Strong unitality of $\Gamma$ and $\Delta$ have been used by Segal \cite{Se}. Strong unitality of $\Theta_n$ has been used by Batanin \cite{Ba} and the author \cite{Be} to decompose an $n$-level tree into a canonical colimit of its linear subtrees. These colimit cocones induce the decomposition of an $n$-dimensional globular \emph{pasting scheme} into globular cells, cf. Leinster \cite{Le}. Strong unitality of $\Omega$ (resp. $\Gamma_\updownarrow$) translates into a canonical decomposition of dendrices into vertex-corollas (resp. of graphices into directed vertex-corollas). In all five cases, these colimit decompositions enter into the Segal model structure for simplical presheaves on $\CC$, cf. \cite{R,CM,HRY1}. In Proposition \ref{strongunitalGamma} below we shall show that the plus construction $\Gamma^+$ is strongly unital as well.

\begin{prp}\label{preservation0}The plus construction $\CC^+$ of a strongly unital hypermoment category $\CC$ is strongly unital.
\end{prp}

\begin{proof}By Proposition \ref{plus}, the Segal core $\CC^+_\Seg$ of the plus construction $\CC^+$ is spanned by $\CC$-vertices $([1],U\rAct A)$ and $\CC$-edges $([0],V)$ where $U,V$ are units of $\CC$ and $A$ is an arbitrary object of $\CC$.

The augmentation $\gamma_\CC:\CC\to\Gamma$ induces a functor $(\gamma_\CC)^+:\CC^+\to\Gamma^+$ which in turn induces functors $\CC^+_{in}\to\Gamma^+_{in}$ and $\CC^+_\Seg\to\Gamma^+_\Seg$ compatible with the respective notions of vertex, edge and incidence relation. Strong unitality of $\CC^+$ amounts to the property that the nerve functor $\CC^+_{in}\to\Pp(\CC^+_\Seg)$ is fully faithful. Since it follows from Proposition \ref{strongunitalGamma} that $\Gamma_{in}^+\to\Pp(\Gamma^+_\Seg)$ is fully faithful, and fully faithful functors are stable under pullback, it suffices to show that the following diagram of functors $$\xymatrix{\CC_{in}^+\ar[r]\ar[d]&\Gamma_{in}^+\ar[d]\\\Pp(\CC^+_\Seg)\ar[r]&\Pp(\Gamma^+_\Seg)}$$induces a fully faithful comparison functor $\CC_{in}^+\to\Pp(\CC^+_\Seg)\times_{\Pp(\Gamma_\Seg^+)}\Gamma_{in}^+$.

For $\CC$-trees $([m],A_\bullet)$ and $([n],B_\bullet)$ a morphism in $\Pp(\CC^+_\Seg)\times_{\Pp(\Gamma_\Seg^+)}\Gamma_{in}^+$ consists of an inert morphism between $\Gamma$-trees (resp. reduced dendrices by Proposition \ref{Gammaplus}), together with compatible mappings between the edge- and vertex-sets of $([m],A_\bullet)$ and $([n],B_\bullet)$. It remains to be shown that such data stems from a uniquely determined inert $\CC$-tree morphism $(\phi,f)$. The simplicial component $\phi:[m]\rIn\,[n]$ is identical to the simplicial component in $\Gamma_{in}^+$. The individual $f_i:A_i\rIn B_{\phi(i)}$ of the natural transformation $f_\bullet:A_\bullet\to B_{\phi(\bullet)}$ may be constructed as follows:

Recall that the edge-set of $([m],A_\bullet)$ is the set of inert morphisms of the form $([0],V)\rIn([m],A_\bullet)$, and the vertex-set of $([m],A_\bullet)$ is the set of inert morphisms $([1],U\rAct A)\rIn([m], A_\bullet)$. They are taken (by the $\Pp(\CC_\Seg^+)$-component) to edges, resp. vertices of $([n],B_\bullet)$. For $i=0$, since $A_0$ is a unit object, the map $A_0\rIn B_{\phi(0)}$ represents the image of the root-edge of $([m],A_\bullet)$ in $([n],B_\bullet)$, and is thus uniquely determined. Assume inductively that we have constructed $A_{i-1}\rIn B_{\phi(i-1)}$. To complete the undotted diagram $$\xymatrix{A_i\ar@{ >.>}[r]^{f_i}&B_{\phi(i)}\\A_{i-1}\ar[u]|+\ar@{ >->}[r]_{f_{i-1}}&B_{\phi(i-1)}\ar[u]|+}$$observe that the edge-set (i.e. the set of elementary inert subobjects) of $A_i$ admits a partition into edge-subsets corresponding to the different vertices of height $i-1$. This partition is induced by the active map $A_{i-1}\rAct A_i$, cf. Proposition \ref{inducedpartition} and Section \ref{Gammavsdendrixinert}. Therefore, each of these edge-subsets corresponds to an inert subobject of $A_i$ taken to $B_{\phi(i)}$ by an inert map. This is done in a compatible way with vertices and incidence relations. The individual edges of $A_i$ are thus also taken compatibly to $B_{\phi(i)}$. By strong unitality of $\CC$ these inert maps glue together and define a uniquely determined inert map $f_i:A_i\rIn B_{\phi(i)}$ rendering the square commutative. By induction, we get the required $\CC$-tree morphism $(\phi,f)$.\end{proof}

\begin{dfn}A unital hypermoment category is called \emph{extensional} if elementary inert morphisms admit pushouts along active morphisms, and these pushouts are inert.
\end{dfn}

These pushouts exist in $\Gamma$ (they are dual to pullbacks of partial identities). They exist in $\Delta$ as well, and using the wreath product, in $\Theta_n$ too. A direct inspection shows that they exist in $\Omega$ and $\Gamma_\updownarrow$, cf. Hackney \cite{H}. In all these cases, the pushouts are preserved under the augmentation. Moreover, inert morphisms are generated (under composition and pushout) by inert morphisms having unital or nil-domain so that we get existence of pushouts of general inert morphisms along active morphisms. This extensionality property of an active/inert factorisation system is dual to what is known in computer science literature as a \emph{modality}.

Note that every \emph{extensionality pushout square}$$\xymatrix{A\ar@{ >.>}[r]&B\circ_\alpha A\\U\ar[u]|+\ar@{ >->}[r]_\alpha&B\ar@{ .>}[u]|+}$$has parallel inert and active morphisms as depicted. Indeed, the upper horizontal morphism is inert by definition, while the right vertical morphism is active because the left part of any orthogonal factorisation system is stable under pushout.

\begin{rmk}It is worthwhile noting that extensionality of $\Delta$ is the key ingredient of the theory of \emph{decomposition spaces} of G\'alvez-Kock-Tonks (cf. \cite{GKT}). A decomposition space is a simplicial presheaf on $\Delta$ taking the extenionsality pushout squares of $\Delta$ to homotopy pullback squares in simplicial sets. In particular, if the decomposition space is discrete, we get genuin pullback squares in sets.

Extensionality of a hypermoment category $\CC$ can be reformulated as follows: since $\Gamma^\op$ is the category $\FinSet_*$ of finite based sets, cf. Examples \ref{examples}a, the augmentation $\gamma_\CC:\CC\to\Gamma$ induces a \emph{cardinality presheaf} $\gamma_\CC^\op:\CC^\op\to\FinSet_*$. The hypermoment category $\CC$ is then extensional if and only if this cardinality presheaf $\gamma_\CC^\op$ is a discrete decomposition space in the aforementionned sense. This condition is weaker than being a discrete Segal presheaf as defined after Definition \ref{strongunitality}. Note that the cardinality presheaf $\gamma_\Delta^\op$ is a simplicial model for the \emph{circle} yielding Segal's delooping machine (cf. \cite{Se}), while the cardinality presheaf $\gamma_{\Theta_n}^\op$ is a $\Theta_n$-model for the \emph{$n$-sphere} yielding the author's $n$-fold delooping machine (cf. \cite{Be2}).\end{rmk}

The following reinforcement of extensionality should be compared with the \emph{extendable algebraic patterns} of Chu-Haugseng \cite[Definition 7.7]{CH}.

\begin{dfn}\label{strongext}An extensional hypermoment category is called \emph{strongly extensional} if for every family of active morphisms $(f_\al:U_\al\rAct B_\al)_{\al\in\el(A)}$ indexed by the elementary subobjects $\al:U_\al\rIn A$ of $A$, there is an essentially unique active morphism $f:A\rAct B$ such that the active part of $f\circ\al$ is $f_\alpha$ for all $\al\in\el(A)$.\end{dfn}

A unital hypermoment category $\CC$ is strongly extensional provided $\CC$ is extensional, strongly unital, and pushing forward along any active morphism $f:A\rAct B$ takes the density colimit cocone of $A$ to a colimit cocone of $B$ (in $\CC_{in}$). This property holds in all examples we have considered so far, except for the hypermoment categories arising as plus construction where a weaker condition holds.

Observe that the objects $([n],A_\bullet)$ of $\CC^+$, the $\CC$-trees, are graded by their height $\hgt([n],A_\bullet)=n$. The elementary subobjects $\alpha:([1],U\rAct A)\rIn([n],A_\bullet)$, i.e. the vertices of $([n],A_\bullet)$, are also graded by their height $\hgt(\al)$, formally defined by the image $[\hgt(\al),\hgt(\al)+1]$ of the interval $[1]$ under the inert morphism $\alpha$. Both gradings are the natural one's in the special case $\CC^+=\Gamma^+$, cf. Appendix \ref{dendrixappendix}.

We call $\CC^+$ \emph{coherently extensional} if the condition of Definition \ref{strongext} is only required for \emph{coherent families} $(f_\alpha:\underline{U}_\al\rAct \underline{B}_\al)_{\al\in\el(A)}$, i.e. those families which satisfy that $\hgt(\al)=\hgt(\be)$ implies $\hgt(\underline{B}_\al)=\hgt(\underline{B}_\be)$. In the special case $\CC^+=\Gamma^+$ this coherence condition ensures the existence of $\underline{A}\rAct \underline{B}$ inside $\Gamma^+\simeq\Omega_r$ which otherwise would only exist in $\Omega$, cf. Appendix \ref{Gammavsdendrixactive} and Remark \ref{insertion}.

\begin{prp}\label{preservation}Let $\,\CC$ be a unital hypermoment category.
\begin{itemize}\item[(a)]If $\,\CC$ is extensional then so is $\,\CC^+$;\item[(b)]If $\gamma_\CC$ preserves the extensionality pushout squares then so does $(\gamma_\CC)^+$.\item[(c)]If $\,\CC$ is strongly extensional then $\CC^+$ is coherently extensional.\end{itemize}\end{prp}

\begin{proof}(a) We have to construct a pushout of the span $$([m],B_\bullet)\lAct([1],U\rAct A)\rIn([n],A_\bullet)$$ in $\CC^+$. The inert morphism on the right induces the right square below$$\xymatrix{B_m\ar@{=}[r]&A\ar@{ >->}[r]&A_{i+1}\\B_0\ar[u]|+\ar@{=}[r]&U\ar[u]|+\ar@{ >->}[r]&A_i\ar[u]|+}$$while the total composition of the $\CC$-tree $([m],B_\bullet)$ on the left may be identified with the left vertical morphism. By extensionality of $\CC$, we can thus push forward along $U\rIn A_i$ the $m$-simplex $B_0\rAct\cdots\rAct B_m$. This defines a $\CC$-tree $$([n+m-1],A_0\rAct\cdots A_i\rAct A_i^{(1)}\cdots\rAct A_i^{(m-1)}\rAct A_{i+1}\cdots\rAct A_n)$$realising the required pushout of the span.

(b) Since extensionality pushout squares are constructed in the same way in $\CC^+$ and in $\Gamma^+$, they are preserved under $(\gamma_\CC)^+$ whenever thay are so under $\gamma_\CC$.

(c) We carry out the construction for each level of $\underline{A}$, i.e. for each active morphism $A_i\rAct A_{i+1}$, separately. A coherent family of $\CC$-trees indexed by the vertices of $\underline{A}$ restricts to a family of $\CC$-trees of \emph{same} height $m_i$ for each elementary subobject of $A_i$, i.e. vertex of $\underline{A}$ of height $i$. Using strong extensionality of $\CC$, we construct level by level an $m_i$-simplex $A_i^{(0)}\rAct A_i^{(1)}\rAct\cdots\cdots\rAct A_i^{(m_i)}=A_{i+1}$ which has the required property with respect to the vertices of height $i$ in $\underline{A}$.

Concatenating these $m_i$-simplices for $i=0,\dots,n-1,$ yields the required $\CC$-tree $\underline{B}$ which comes equipped with an obvious active $\CC$-tree morphism $\underline{A}\rAct\underline{B}$ enjoying the required property with respect to all vertices of $\underline{A}$.\end{proof}

\begin{rmk}\label{insertion}The extensionality pushout $([n],A_\bullet)\circ_\alpha([m],B_\bullet)$ constructed in (a) may be viewed as the result of \emph{inserting} the $\CC$-tree $([m],B_\bullet)$ into the vertex represented by $\alpha:([1],U\rAct A)\rIn([n],A_\bullet)$. Indeed, after application of $(\gamma_{\CC})^+:\CC^+\to\Gamma^+$, the pushout realises the tree-insertion of the respective reduced dendrices in $\Gamma^+\simeq\Omega_r$, cf. Appendix \ref{dendrixappendix}. The reader should however keep in mind that tree-insertion in $\Omega_r$ differs from tree-insertion in $\Omega$ because inserting a reduced dendrix into the vertex of another reduced dendrix may result in a non-reduced dendrix inside $\Omega$. It is a pleasant feature of the categorical pushout that it performs precisely the right thing to correct this failure.

The extensionality pushout constructed in (c) corresponds to a simultaneous insertion of the given coherent family of $\CC$-trees into all vertices of $\underline{A}$. The coherence condition guarantees that the result of this insertion process corresponds after application of $(\gamma_\CC)^+:\CC^+\to\Gamma^+$ to a geometric insertion of the corresponding reduced dendrices (cf. \cite[Chapter IV]{BB}). In Section \ref{Gammavsdendrixactive} general active morphisms $\phi:S\rAct T$ of dendrices are described as families $(T_\alpha)_{\alpha\in V(S)}$ of subdendrices of $T$, indexed by the vertices of $S$ and fulfilling a weaker coherence condition. If $S$ and $T$ are reduced dendrices then $\phi$ is a reduced dendrix morphism precisely when the subdendrices $T_\alpha,T_\beta$ of $T$ have same height whenever the vertices $\alpha,\beta$ have same height in $S$.

The categorical pushout $S\circ_\alpha T_\alpha$ in $\Omega_r$ inserts $T_\alpha$ into the vertex $\alpha$ of $S$ and simultaneously a \emph{stretched} corolla $\overline{C}_\beta$ of same height as $T_\alpha$ into each vertex $\beta$ of $S$ of same height as $\alpha$. This produces a reduced dendrix $T$ representing the categorical pushout $S\circ_\alpha T_\alpha$ in $\Omega_r$, and this is the way the categorical pushout (a) should be thought of. Note that the stretched corolla $\overline{C}_\beta$ is obtained from the original corolla $C_\beta$ in $S$ by replacing its \emph{input edges} with linear trees of appropriate height.\end{rmk}

\subsection{Monadicity}Our final goal is to show that for strongly extensional hypermoment categories $\CC$, the forgetful functor from $\,\CC$-operads to $\,\CC$-collections is monadic. Our proof uses the equivalence between $\CC$-operads and $\CC^+$-monoids as well as an explicit formula for the free $\CC^+$-monoid generated by a $\CC$-collection.

Similar monadicity results have been obtained by Getzler \cite[Corollary 2.8]{G}, Kaufmann-Ward \cite[Theorems 1.5.3 and 1.5.6]{KW}, Chu-Haugseng \cite[Section 8]{CH} and Batanin-Markl \cite[Section 3]{BaM3}.

A key ingredient is the existence of a simultaneous $\CC$-tree insertion even if the family of $\CC$-trees to be inserted is not coherent. The trick is that any family can be replaced with a coherent family in a canonical and optimal way. We need a notion of \emph{$\CC$-tree contraction} formally inverse to \emph{$\CC$-tree stretching}.

\begin{dfn}\label{contraction}An active $\CC$-tree morphism $(\phi,f):([m],A_\bullet)\rAct([n],B_\bullet)$ is called a $k$-\emph{contraction} if $m=n+k$ and $\phi(n)=\phi(n+1)=\cdots=\phi(m)$ and $A_n\overset{=}{\rAct}A_{n+1}\overset{=}{\rAct}\cdots\overset{=}{\rAct}A_m$ and $f_i:A_i\overset{=}{\to} B_{\phi(i)}$ for $i=0,\dots,n-1$.\end{dfn}

The simplicial operator $\phi$ is the composite of $k$ elementary degeneracy operators of last index.  Each $\CC$-tree contraction $\underline{B}\rAct\underline{A}$ has an inert section $\underline{A}\rIn\underline{B}$ whose first component is a composite of simplicial face operators of last index. The vertices of $\underline{B}$ not contained in the image of $\underline{A}\rIn\underline{B}$ are \emph{effaceable}, i.e. they come equipped with a contraction to a nil-$\CC$-tree $([0], U)$ for some unit $U$ of $\CC$. Note that every object $A$ of $\CC$ induces a unit-$\CC$-tree $([1],U_A\rAct A)$, but only those unit-$\CC$-trees of the form $([1],1_U)$ come equipped with a contraction to a nil-$\CC$-tree. They corepresent effaceable vertices.

\begin{lma}\label{widepushout}Let $\CC$ be a strongly extensional hypermoment category and let $\underline{A}$ be a $\,\CC$-tree.
For any family $(\underline{U}_\alpha\rAct \underline{A}_\alpha)_{\al:\underline{U}_\alpha\!\rIn\!\underline{A}}$ of $\,\CC$-trees indexed by the vertex set of $\underline{A}$, there is an active $\CC$-tree morphism $f:\underline{A}\rAct\underline{B}$ such that

\begin{enumerate}\item for each $\alpha:\underline{U}_\alpha\rIn \underline{A}$, the given $\CC$-tree $\underline{U}_\alpha\rAct \underline{A}_\alpha$ factors through a $k_\alpha$-contraction $r_\alpha:\underline{B}_\alpha\rAct\underline{A}_\alpha$ inducing a commutative diagram\begin{diagram}[small]\underline{A}&\rAct^f&\underline{B}&\\\uIn^\alpha&&\uIn\\\underline{U}_\alpha&\rAct_{f_\alpha}&\underline{B}_\alpha&\rAct_{r_\alpha}\underline{A}_\alpha\end{diagram}
\item the contraction degree $k_f=\sum_{\alpha:\underline{U}_\alpha\!\rIn\!\underline{A}}k_\alpha$ of $f:\underline{A}\rAct\underline{B}$ is minimal among the contraction degrees
 $k_{\tilde{f}}$ of all $\tilde{f}:\underline{A}\rAct \tilde{\underline{B}}$ fulfilling (1);\item the $\CC$-tree $\underline{B}$ is up to isomorphism uniquely determined by (1) and (2).\end{enumerate}\end{lma}

\begin{proof}Recall that the vertices $\alpha:\underline{U}_\alpha\rIn\underline{A}$ are graded by height so that for a given height $h$ in $\underline{A}$ we get a positive integer $m_h=\max\{\hgt(\underline{A}_\alpha)\,|\,\hgt(\alpha)=h\}$. We now define for a fixed height $h$ in $\underline{A}$, $\CC$-trees $B_\alpha$ of height $m_h$ obtained from $\underline{A}_\alpha$ by a $k_\alpha$-stretching where $k_\alpha=m_h-\hgt(\underline{A}_\alpha)$. Note that we get corresponding $k_\alpha$-contractions $\underline{B}_\alpha\rAct \underline{A}_\alpha$ in the sense of Definition \ref{contraction}.

If we do this for all vertices of $\underline{A}$ we get a coherent family of $\CC$-trees $\underline{B}_\al$ together with $k_\al$-contractions $r_\al:\underline{B}_\al\rAct \underline{A}_\alpha.$ Consequently, an application of Proposition \ref{preservation}c yields an active $\CC$-tree morphism $f:\underline{A}\rAct\underline{B}$ completing diagram (1) above. It is then straighforward to check that properties (2) and (3) hold as well.\end{proof}

\begin{thm}\label{monadicity}For any strongly extensional hypermoment category $\CC$ and any cocomplete closed symmetric monoidal category $\EE$, the forgetful functor from $\,\CC$-operads to $\,\CC$-collections in $\EE$ is monadic.\end{thm}

\begin{proof}Theorem \ref{BaezDolanplus} shows that the categories of $\CC$-operads and $\CC^+$-monoids are equivalent. Under this equivalence the forgetful functor corresponds to the functor which takes a $\CC^+$-monoid $\Yy:(\CC^+_{act})^\op\to\EE$ to its restriction $\UU_+(\Yy):(\CC^+_{unit})^\op\to\EE$ to the full subcategory spanned by the unit $\CC$-trees, cf. the proof of Theorem \ref{BaezDolanplus}. It is thus sufficient to show that $\UU_+$ is monadic.

For $\Xx:(\CC^+_{unit})^\op\to\EE$ we define a $\CC^+$-monoid $\FF_+(\Xx):(\CC^+_{act})^\op\to\EE$ as follows. For each unit $\CC$-tree $\underline{U}$ let $\FF_+(\Xx)(\underline{U})$ denote the coend$$\FF_+(\Xx)(\underline{U})=I^{\otimes\CC^+_{act}(\underline{U},-)}\otimes_{\Aut_{\underline{U}}(-)}S(\Xx,-)\text{ where }S(\Xx,\underline{B})=\bigotimes_{\underline{V}_\beta\rIn\underline{B}}\Xx(\underline{V}_\beta)$$ and $I$ denotes the monoidal unit of $\EE$. The coend exists by cocompleteneness.

Informally, $\FF_+(\Xx)(\underline{U})$ is a coproduct indexed by $\CC$-trees whose total composition belongs to the active component of $\underline{U}$, where the summands are tensor products of values of $\Xx$ according to the vertices of the indexing $\CC$-tree. The coend identifies automorphisms induced by $\CC$-tree symmetries with those induced by $\Xx$.

In order to endow $\FF_+(\Xx)$ with a $\CC^+$-monoid structure, we set (cf. Section \ref{monoidsection})$$\FF_+(\Xx)(\underline{A})=\bigotimes_{\underline{U}_\alpha\rIn \underline{A}}\FF_+(\Xx)(\underline{U}_\alpha).$$ The data of a $\CC^+$-monoid consists of suitable maps $\FF_+(\Xx)(\underline{A})\to\FF_+(\Xx)(\underline{U})$ for all active $\CC$-tree morphisms $\underline{U}\rAct\underline{A}$ with unital domain.

Observe that the value $\FF_+(\Xx)(([0],V))$ at a nil-$\CC$-tree is the monoidal unit $I$ of $\EE$ and that for a unit-$\CC$-tree $([1],1_V)$ with effaceable vertex, the value $\FF(\Xx)(([1],1_V))$ contains a distinguished element represented by the identity $1_{\underline{V}}$. There is thus a map $\FF_+(\Xx)(([0],V))\to\FF_+(\Xx)([1],1_V)$ corresponding to this distinguished element. Consequently, any $\CC$-tree contraction $\underline{B}_\alpha\rAct\underline{A}_\alpha$ induces a canonical map $S(\Xx,\underline{A}_\alpha)\to S(\Xx,\underline{B}_\alpha)$ obtained by tensoring $S(\Xx,\underline{A}_\alpha)$ with maps $I\to \Xx(\underline{V})$ for effaceable vertices $\underline{V}\rIn\underline{B}_\alpha$ not contained in the image of $\underline{A}_\alpha\rIn\underline{B}_\alpha$.

We now construct the required maps $\FF_+(\Xx)(\underline{A})\to\FF_+(\Xx)(\underline{U})$ for active $\CC$-tree morphisms $\underline{U}\rAct\underline{A}$. Since by closedness of $\EE$ coproducts distribute over tensor products, a typical element of $\FF_+(\Xx)(\underline{A})$ is associated to a family of $\CC$-trees $\underline{U}_\alpha\rAct\underline{A}_\alpha$ indexed by vertices $\underline{U}_\alpha\rIn\underline{A}$. By Lemma \ref{widepushout} this family defines an essentially unique active $\CC$-tree morphism $\underline{A}\rAct\underline{B}$ inducing a canonical map $$\bigotimes_{\underline{U}_\alpha\rIn\underline{A}}S(\Xx,\underline{A}_\alpha)\longrightarrow\bigotimes_{\underline{U}_\alpha\rIn\underline{A}}S(\Xx,\underline{B}_\alpha)$$whose right hand side may be identified with $S(\Xx,\underline{B})$ because the vertex set of $\underline{B}$ is the disjoint union of the vertex sets of the $\underline{B}_\alpha$ by Lemma \ref{active=surj}. Together with $\underline{U}\rAct\underline{A}\rAct\underline{B}$ this represents an element of $\FF_+(\Xx)(\underline{U})$.

The functoriality of $\FF_+(\Xx)$ with respect to active $\CC$-tree morphisms follows from the fact that (via the above constructed map) $\FF_+(\Xx)(\underline{A})$ may be identified with a direct summand of $\FF_+(\Xx)(\underline{U})$, namely the one associated with those active $\CC$-tree morphisms $\underline{U}\rAct\underline{B}$ which factor through $\underline{U}\rAct\underline{A}$.

Any map of $\CC^+$-monoids $\FF_+(\Xx)\to\Yy$ restricts to $\UU_+\FF_+(\Xx)\to\UU_+(\Yy)$. Precomposing with the canonical unit map $\Xx\to\UU_+\FF_+(\Xx)$ yields an adjoint map of $\CC$-collections $\Xx\to\UU_+(\Yy)$. Conversely, starting with the latter, we get a map of $\CC^+$-monoids $\FF_+(\Xx)\to\FF_+\UU_+(\Yy)$. Since $\Yy$ is a $\CC^+$-monoid, there is a canonical counit map of $\CC^+$-monoids $\FF_+\UU_+(\Yy)\to\Yy$. The triangular identities follow readily from the definitions as well as the fact that algebras over the monad $\UU_+\FF_+$ may be identified with $\CC^+$-monoids in $\,\EE$.\end{proof}

\begin{cor}\label{extension}For any functor $f:\CC\to\DD$ of strongly extensional hypermoment categories the restriction functor $f^*:\Oper_\DD(\EE)\to\Oper_\CC(\EE)$ admits a left adjoint extension functor $f_!:\Oper_\CC(\EE)\to\Oper_\DD(\EE)$.\end{cor}

\begin{proof}This is an immediate consequence of the adjoint lifting theorem because the functor of presheaf categories $f^*:\Coll_\DD(\EE)\to\Coll_\CC(\EE)$ has a left adjoint.\end{proof}

\begin{rmk}Theorems \ref{BaezDolanplus} and \ref{monadicity} together with their proofs provide an explicit formula for the $\CC$-operad freely generated by a $\CC$-collection. In the special case $\CC=\Gamma$ we recover the classical formula for the free symmetric operad generated by a symmetric collection with the difference that in our formula only reduced dendrices occur, while classically general dendrices are used. This difference is handled by a stretching/contraction process and reflects the fact that symmetric operads can be represented as well as $\Gamma^+$-monoids as well as $\Omega$-monoids.

As an application of Corollary \ref{extension} we recover Batanin's \emph{symmetrisation} functor $(\gamma_{\Theta_n})_!:\Oper_{\Theta_n}(\EE)\to\Oper_\Gamma(\EE)$ turning an $n$-operad into a symmetric operad, cf. Remark \ref{operadexamples}c and \cite{Ba2}. One of the main results of \cite{Ba2} states that simplicial algebras over a cofibrant replacement of the terminal $\Theta_n$-operad are models for $n$-fold loop spaces because the symmetrisation of such a replacement is an $E_n$-operad. On the other hand, the author showed in \cite{Be2} that $(\Theta_n)_\infty$-monoids are also models for $n$-fold loop spaces. This is certainly not a coincidence and suggests that an analogous result holds for any strongly extensional hypermoment category.\end{rmk}

\appendix

\section{Operadic categories from rigid hypermoment categories}\label{operadicappendix}

In this appendix we make explicit the structure of operadic category carried by the dual of the active part of a rigid hypermoment category. This produces valuable examples of operadic categories. For our convenience we actually show that the active part of a rigid hypermoment category is a \emph{co}operadic category. Let us review its definition, obtained by dualising the definition of an operadic category.

\subsection{Cooperadic categories}

A category $\CC$ is said to be \emph{cooperadic} (cf. Batanin-Markl \cite[Part 1,Section 1]{BaM}) if $\CC$ comes equipped with

\begin{enumerate}\item designated initial objects in each connected component,
\item a cardinality functor $\gamma_\CC:\CC\to\Gamma_{act}$,
\item for each pair $(f,\alpha)$ consisting of a morphism $f:A\to B$ and an element $\alpha\in\gamma_\CC(A)$, there is given an object $B_\alpha$ in $\CC$, called the \emph{cofibre} of $f$ over $\alpha$, such that $\gamma_\CC(B_\alpha)=\gamma_\CC(f)(\alpha)$, and the following five axioms hold:
\end{enumerate}

\vspace{1ex}

\emph{Axiom (i):} the cardinality of initial objects is one.\vspace{1ex}

\emph{Axiom (ii):} the cofibres of identity morphisms are designated initial.\vspace{1ex}

\emph{Axiom (iii):} for each $\alpha\in\gamma_\CC(A)$ and each $f:B\to C$ under $A$ there is a morphism $f_\alpha:B_\alpha\to C_\alpha$ depending functorially on $f$ (i.e. $(gf)_\alpha=g_\alpha f_\alpha$).\vspace{1ex}

\emph{Axiom (iv):} for each morphism $f:B\to C$ under $A$, and elements $\alpha\in\gamma_\CC(A)$, $\beta\in\gamma_\CC(A\to B)(\alpha)$, the cofibre of $f$ over $\beta$ coincides with the cofibre of $f_\alpha$ over $\beta$.\vspace{1ex}

\emph{Axiom (v):} consider the following commutative diagram$$\xymatrix{A\ar[rr]^k\ar[dd]\ar[ddrr]&&A'\ar[dd]^h\ar@{-}[ld]_f\\&\ar[ld]&\\B\ar[rr]_g&&C}$$and assume given $\alpha\in\gamma_\CC(A)$ and $\alpha'\in\gamma_\CC(k)(\alpha)$. According to Axiom (iii) we get a commutative triangle $$\xymatrix{&(A')_{\alpha}\ar[ld]_{f_{\alpha}}\ar[rd]^{h_{\alpha}}&\\B_{\alpha}\ar[rr]_{g_{\alpha}}&&C_{\alpha}}$$where by assumption $\alpha'\in\gamma_\CC(A'_\alpha)$. It is then required that $(g_\alpha)_{\alpha'}=g_{\alpha'}$.

\subsection{Rigid hypermoment categories}

Recall that a unital hypermoment category is said to be rigid if every isomorphism is an automorphism, and every automorphism acts trivially on active morphisms with unital domain.

\begin{lma*}\label{uniqueactive}In a rigid hypermoment category, every morphism with unital domain has a uniquely determined active part. In particular, every object receives a uniquely determined active morphism from a uniquely determined unit.\end{lma*}

\begin{proof}Since isomorphisms are automorphisms, two active/inert factorisations of the same morphism can only differ by an automorphism of the middle object. If the domain is unital such an automorphism acts trivially so that the active part is uniquely determined. The second assertion follows then from the definition of a unital hypermoment category.\end{proof}

Each connected component of the active part $\CC_{act}$ has thus a uniquely determined initial object $U_\alpha$ which will be the designated initial object. By restriction we get a cardinality functor $\gamma_\CC:\CC_{act}\to\Gamma_{act}$ with values in $\Gamma_{act}$. For each $f:A\rAct B$ and each element $\alpha:\11\rIn\gamma_\CC(A)$ there is a unique inert lift $U_\alpha\rIn A$ in $\CC$. We will tacitely identify elements of $\gamma_\CC(A)$ with their inert lifts. The active/inert factorisation of $U_\alpha\rIn A\rAct B$ yields $U_\alpha\rAct B_\alpha\rIn B$. By Lemma \ref{uniqueactive}, the active part $U_\alpha\rAct B_\alpha$ is uniquely determined so that its target $B_\alpha$ can serve as the \emph{cofibre} of $f$ over $\alpha$. We thus get the underlying data of a cooperadic category.

\subsection{Proof of Proposition \ref{operadic}}\label{proofoperadic}
Let us now check that for each rigid hypermoment category the active part satisfies the five axioms of a cooperadic category.\vspace{1ex}

Axioms (i) and (ii) follow immediately from the definitions. For Axiom (iii) consider the following diagram

$$
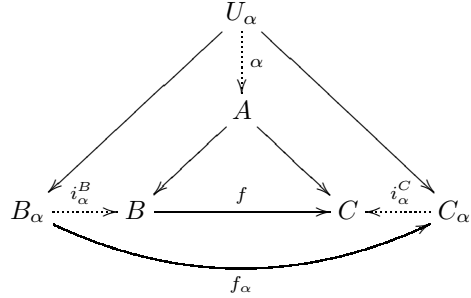
\xymatrix{&&U_\alpha\ar@{.>}[d]^{\alpha}\ar[lldd]\ar[rrdd]&&\\&&A\ar[ld]\ar[rd]&&\\
B_\alpha\ar@/_{5ex}/[rrrr]_{f_\alpha}\ar@{.>}[r]^{i^B_\alpha}&B\ar[rr]^{f}&&C&\ar@{.>}[l]_{i^C_{\alpha}}C_\alpha}$$

\noindent in which the dotted (resp. undotted) arrows are inert (resp. active). It suffices to apply the active/inert factorisation system to the composite morphism $fi_\alpha^B:B_\alpha\rIn B\rAct C$ to get $i_\alpha^Cf_\alpha:B_\alpha\rAct C_\alpha\rIn C$. The functoriality of the assignment $f\mapsto f_\alpha$ is a consequence of the uniqueness of the active/inert factorisation system, due to rigidity.

For Axiom (iv) observe that for $\beta:U\rIn B_\alpha$ the active/inert factorisation of the composite morphism $fi_\alpha^B\beta$ can be achieved in two steps, first replacing $i_\alpha^Bf$ with $i_\alpha^Cf_\alpha$, then applying the factoriation system to $f_\alpha\beta$. Finally, for Axiom (v), the argument just given for Axiom (iv), shows that $(g_\alpha)_{\alpha'}$ and $g_{\alpha'}$ have same sources and targets. The argument actually also shows that the two morphisms coincide.

This completes the proof of Proposition \ref{operadic}.\qed

\subsection{Rigidification}\label{rigidify}The moment categories $\Gamma$, $\Delta$, and $\Theta_n$ of Examples \ref{examples} are rigid hypermoment categories. We get as associated operadic categories respectively (a skeleton of) the category of finite sets, the category of finite ordered sets, and the category of $n$-level trees of Batanin \cite{Ba2}. The latter embeds canonically into Joyal's category of combinatorial $n$-disks \cite{J} which is isomorphic to $\Theta_n^\op$, cf. \cite{Be2}.

The hypermoment categories $\Omega$ and $\Gamma_\updownarrow$ of Section \ref{ix} are not rigid. However, there are \emph{planar} versions of $\Omega$ and $\Gamma_\updownarrow$ which are rigid. There is also a different, purely combinatorial way of rigidifying $\Omega$ (resp. $\Gamma_\updownarrow$) by endowing a dendrix (resp. graphix) with linear orderings of the incoming/outgoing leaves, and also with linear orderings of the incoming/outgoing edges of each vertex. If the active morphisms are required to preserve the linear orderings of the leaves and also to convert the linear ordering of a vertex-corolla to the linear ordering of the leaves of the expanded tree, then we get a rigid hypermoment category with associated operadic category. It would be most useful to know \emph{which} unital hypermoment categories admit such combinatorial rigidifications and, if so, how to construct them in an intrinsic way.


\section{Embedding $\Gamma^+$ into the dendroidal category $\Omega$}\label{dendrixappendix}

In this appendix the plus construction $\Gamma^+$ is identified with a hypermoment subcategory $\Omega_r$ of $\Omega$. The latter consists of \emph{reduced} dendrices and \emph{reduced} dendrix morphisms, cf. Section \ref{ix}. This embedding $\Gamma^+\inc\Omega$ is interesting for two reasons.

First, although implicit in the comparison of two notions of $\infty$-operad by Chu-Haugseng-Heuts \cite{CHH}, the hypermoment subcategory $\Omega_r$ has so far not attracted much attention despite of the surprising fact that from a homotopical point of view $\Omega_r$ does the same job as $\Omega$, cf. \cite[Theorem 5.1]{CHH} and Remark \ref{CHH}.

Second, each unital hypermoment category $\CC$ induces a functor of plus constructions $(\gamma_\CC)^+:\CC^+\to\Gamma^+$. In particular, the objects of $\CC^+$, the so-called $\CC$-trees, may be viewed as reduced dendrices equipped with further structure: a colouring of the edges by units of $\CC$. This yields a geometric perspective on our plus construction.

Via the isomorphism $\Gamma^+\cong\tilde{\Delta}^1_\FF$ of Remark \ref{CHH}, our embedding $\Gamma^+\inc\Omega$ is induced by the functor $\tau:\Delta_\FF^1\to\Omega$ constructed in \cite[Section 4]{CHH}. Taking the image of $\tau$ eliminates certain ``degenerate'' objects from the original $\Delta_\FF^1$. While Chu-Haugseng-Heuts rely on Kock's combinatorial description \cite{K0} of the dendroidal category $\Omega$, we stick here as closely as possible to the original definition of Moerdijk-Weiss \cite{MW}.

\subsection{Reduced dendrices and $\Gamma$-trees}--\label{Gammavsdendrix}\vspace{1ex}

Let $([m],A_0\rAct\cdots\rAct A_m)$ be a $\Gamma$-tree as in Definition \ref{hyperdefinition}. We now associate to such a $\Gamma$-tree in an informal way a dendrix, see Figure \ref{dendrix}. The objects of $\Gamma$ are viewed as finite sets. The elements of $A_i$ are the edges of height $i$ of the dendrix associated with $([m],A_\bullet)$. The active morphism $A_i\rAct A_{i+1}$ encodes which edges of $A_{i+1}$ are supported by a given edge of $A_i$. If an element of $A_i$ indexes the empty subset of $A_{i+1}$ then the corresponding edge gets a stump. If $A_m\not=\underline{0}$ then the dendrix has as many leaves as there are elements in $A_m$.

We get in this way a dendrix which is closed if and only if $A_m=\underline{0}$ and otherwise has all its leaves at maximal height, and no edges with stump at maximal height. This is the definition of a \emph{reduced} dendrix. Conversely, every reduced dendrix determines a unique $\Gamma$-tree to which it is associated. Notice that this correspondence is compatible with the respective notions of vertex.

\begin{figure}
$$\xymatrix@=0.6cm@R=0.3cm{&\ar@{-}[ddr]^1&\ar@{-}[dd]|2&\ar@{-}[ddl]_3&&&\\&&&&&\underline{3}&\\\bullet\ar@{-}[ddr]^1 &&\bullet  \ar@{-}[ddl]_2&&&&\\&&&&&\underline{2}\ar[uu]|{+}_{(\emptyset,\{1,2,3\}\}}&\\& \bullet\ar@{-}[ddr]^1 & &\bullet\ar@{-}[ddl]_2&&
 &\\&&&&&\underline{2}\ar[uu]|{+}_{(\{1,2\},\emptyset)}&\\ & &\bullet \ar@{-}[dd]^1& & &&\\&&&&&\underline{1}\ar[uu]|{+}_{(\{1,2\})}&\\&&&&&&}$$
\caption{Reduced dendrix and corresponding $\Gamma$-tree}
\label{dendrix}
\end{figure}
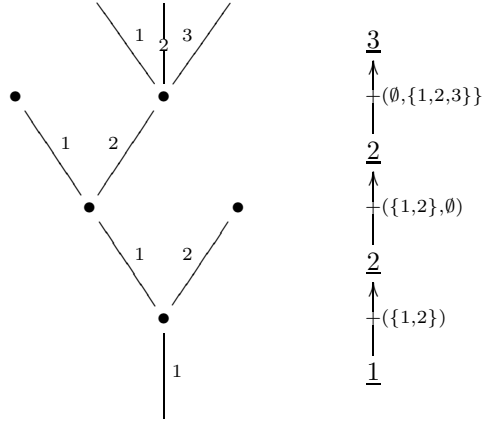

\subsection{Combinatorial description of inert dendrix morphisms.}\label{Gammavsdendrixinert}In order to give a rigorous definition of the previous correspondence we have to decribe first the inert part of the dendroidal category in purely combinatorial terms.

A \emph{dendrix} $S$ is given by a finite set $E(S)$ of \emph{edges} and a finite set $V(S)$ of \emph{vertices}, equipped with an incidence relation such that each edge is incident to at most two vertices, and each vertex has a distinguished incident edge. This distinguished edge is called \emph{outgoing} while the other incident edges of the same vertex are called \emph{incoming}. Two axioms must be satisfied: (1) all \emph{inner} edges (i.e. those incident to two vertices) are outgoing from one and incoming into the other incident vertex; (2) there is a unique \emph{outer} (i.e. non-inner) edge, the so-called \emph{root}, such that each edge is linked to the root through an oriented edge-path. The outer non-root edges are called \emph{leaves}. We assume that dendrices are non-empty, and that there is a dendrix $|$ (the free-living edge) with a single edge but no vertices. For each non-negative integer $k\geq 0$ there is a dendrix $C_k$ with a single vertex having $k$ incoming edges and one outgoing edge. This dendrix $C_k$ is called a \emph{corolla} of arity $k$. More generally, a vertex of a dendrix is said to a have \emph{arity} $k$ if it has precisely $k$ incoming edges.

An inert morphism $S\rIn T$ between dendrices $S,T$ is given by a pair of maps $(V(S)\to V(T),E(S)\to E(T))$ respecting the incidence relations, the distinguished outgoing edges, and the vertex arities. The inert part $\Omega_{in}$ of the dendroidal category of Moerdijk-Weiss \cite{MW} consists precisely of dendrices and inert morphisms between them. Recall from Definition \ref{strongunitality} that the \emph{Segal core} $\Omega_\Seg$ is the full subcategory of $\Omega_{in}$ spanned by the free-living edge $|$ and the corollas $C_k$.

Set-valued presheaves on $\Omega_\Seg$ are called \emph{$\Omega$-graphs}. The functor taking a dendrix $S$ to its $\Omega$-graph $S_*=\Omega_{in}(-,S)$ is a fully faithful functor $\Omega_{in}\to\Pp(\Omega_\Seg)$. This is just a reformulation of \emph{strong unitality} of $\Omega$, cf. Section \ref{Segalcore} and \cite[Section 1.6]{BMW}.

We now define a comparison functor $\Gamma^+_{in}\to\Omega_{in}$. Each $\Gamma$-tree $([m],A_\bullet)$ defines a dendrix $S_{([m],A_\bullet)}=(E_{([m],A_\bullet)}),V_{([m],A_\bullet)})$ by letting the edge-set be the set of inert subobjects of the form $([0],\11)\rIn([m],A_\bullet)$ and the vertex-set be the set of inert subobjects of the form $([1],\11\rAct\kk)\rIn([m],A_\bullet)$ for varying $k\geq 0$. An edge is incident to a vertex when it factors through it. It is \emph{incoming} (resp. \emph{outgoing}) if the first component of $([0],\11)\rIn([1],\11\rAct\kk)$ takes $0$ to $1$ (resp. $0$).

Inert morphisms of $\Gamma$-trees $([m],A_\bullet)\rIn([n],B_\bullet)$ induce (by postcomposition) inert morphisms $S_{([m],A_\bullet)}\rIn S_{([n],B_\bullet)}$ between the associated dendrices.

\begin{prp*}\label{strongunitalGamma}The comparison functor $\Gamma_{in}^+\to\Omega_{in}$ is fully faithful. Its essential image is spanned by the reduced dendrices. In particular, the hypermoment category $\Gamma^+$ is strongly unital (cf. Definition \ref{strongunitality}).\end{prp*}

\begin{proof}Notice first that the comparison functor identifies the Segal core of $\Gamma^+$ with the Segal core of $\Omega$. Indeed, the units $([1],\11\rAct\kk)$ of $\Gamma^+$ are taken to the corollas $C_k$, and the nilobject $([0],\11)$ of $\Gamma^+$ is taken to the free-living edge $|$. Moreover, the unit $([1],\11\rAct\kk)$ has the same automorphism group in $\Gamma^+$ as the corolla $C_k$ in $\Omega$, and there are precisely $k+1$ inert morphisms $([0],\11)\rIn([1],\11\rAct\kk)$ in $\Gamma^+$ corresponding to the $k+1$ edge-inclusions $|\rIn C_k$ in $\Omega$. Between distinct units there are no inert morphisms neither in $\Gamma^+$ nor in $\Omega$. We can thus identify the category $\Pp(\Gamma_\Seg)$ of $\Gamma$-graphs with the category $\Pp(\Omega_\Seg)$ of $\Omega$-graphs getting the following commutative diagram of functors
$$\xymatrix{\Gamma_{in}^+\ar[r]\ar[d]&\Omega_{in}\ar[d]\\\Pp(\Gamma_\Seg)\ar@{=}[r]&\Pp(\Omega_\Seg)}$$in which the right vertical functor is fully faithful. Therefore, the upper horizontal functor is faithful (resp. full) if and only if the left vertical functor is so.

Now, an inert morphism between $\Gamma$-trees is completely determined by what it does to the edges of the associated dendrices. This proves faithfulness. For $\Gamma$-trees $([m],A_\bullet),([n],B_\bullet)$ consider an inert morphism $S_{([m],A_\bullet)}\rIn S_{([n],B_\bullet)}$ between the associated dendrices. The root of $S_{([m],A_\bullet)}$ is taken to an edge of $S_{([n],B_\bullet)}$ which has a well-defined height $h$ inside $S_{([n],B_\bullet)}$ (where the height of the root is $0$). This defines a unique inert map $\phi:[m]\to[n]$ in $\Delta$ taking $0$ to $h$. Moreover, the induced edge-map $E(S_{([m],A_\bullet)})\to E(S_{([n],B_\bullet)})$ defines inert maps $f_i:A_i\rIn B_{\phi(i)}$ in $\Gamma$ for all $i\in[m]$. The pair $(\phi,f)$ defines an inert map $([m],A_\bullet)\rIn([n],B_\bullet)$ in $\Gamma^+$ which by construction is taken to the given inert map $S_{([m],A_\bullet)}\rIn S_{([n],B_\bullet)}$ in $\Omega$. This proves fullness and hence strong unitality of the hypermoment category $\Gamma^+$.

The functor $\Gamma_{in}^+\to\Omega_{in}$ induces on objects the one-to-one correspondence of Section \ref{Gammavsdendrix} so that the essential image of $\Gamma_{in}^+\to\Omega_{in}$ consists precisely of reduced dendrices and inert morphisms between them.\end{proof}

\subsection{Combinatorial description of general dendrix morphisms}\label{Gammavsdendrixactive}

In order to relate $\Gamma$-tree and dendrix morphisms in total generality, we have to make explicit the \emph{free coloured (symmetric) operad} generated by a dendrix. This construction has been discussed at several places in literature, see e.g. \cite{BeMo, MW, W, K0, HM}. For sake of completeness let us review it here.

Each coloured operad $\Oo$ has an underlying $\Omega$-graph $U_\Omega\Oo$ for which $(U_\Omega\Oo)(|)$ is the colour-set of $\Oo$, and $(U_\Omega\Oo)(C_k)$ is the set of $k$-ary operations of $\Oo$. The action of $\Aut(C_k)$ accounts for the symmetries, while the remaining presheaf action of $U_\CC\Oo$ determines inputs and output of the individual operations. This forgetful functor $U_\Omega$ has a left adjoint $F_\Omega$. The associated monad $T_\Omega=U_\Omega F_\Omega$ on the category of $\Omega$-graphs fixes the colour-set and has for any $\Omega$-graph $X$ the following description $$T_\Omega(X)(C_k)=\coprod_{[S]}\Omega_{act}(C_k,S)\times_{\Aut(S)}\Hom_{\Pp(\Omega_\Seg)}(S_*,X)$$where the coproduct ranges over isomorphism classes of dendrices. The relevant case is when $X=T_*$ is the $\Omega$-graph induced by a dendrix $T$. We get as set of free $k$-ary operations $T_\Omega(T_*)(C_k)$ the set of pairs $(S,\sigma)$ consisting of a subdendrix (i.e. inert subobject) $S$ of $T$ and a bijection $\sigma$ between the leaves of $C_k$ and the leaves of $S$. The same construction based on the formalism of tree-polynomials is described by Kock in \cite[Corollary 1.2.10]{K0}.

It turns out that $U_\Omega$ is \emph{monadic} (cf. Weber \cite[Example 2.14]{W}) so that for any dendrices $S,T$ we get the following Kleisli-type description of the morphism-set$$\Omega(S,T)=\Hom_{\Pp(\Omega_\Seg)}(S_*,T_\Omega(T_*))$$based only on the inert part of $\Omega$ and the free-forgetful monad $T_\Omega$ on $\Omega$-graphs. Notice that the inclusion $\Omega_{in}\subset\Omega$ is induced by the \emph{unit} $\eta:id_{\Pp(\Omega_\Seg)}\to T_\Omega$, and composition in $\Omega$ is induced by the \emph{multiplication} $\mu:T_\Omega^2\to T_\Omega$ of the monad.

In more geometrical terms, a dendrix morphism $S\to T$ is represented by a family $(T_\alpha)_{\alpha\in V(S)}$ of subdendrices of $T$, indexed by the vertex-set of $S$, and subject to the following two conditions: (1) the leaves (resp. root) of $T_\alpha$ are in bijection with the incoming (resp. outgoing) edges of $\alpha$, (2) the subdendrices $T_\alpha$ are grafted inside $T$ according to the same pattern as the vertex-corollas $\alpha$ inside $S$.

We can now assign to each $\Gamma$-tree morphism $([m],A_\bullet)\to([n],B_\bullet)$ a dendrix morphism $S_{([m],A_\bullet)}\to S_{([n],B_\bullet)}$ by means of a family $(T_\alpha)_{\alpha\in V(S_{([m],A_\bullet)})}$ in which $T_\alpha$ represents the subdendrix of $S_{([n],B_\bullet)}$ obtained by factorising$$([1],\11\rAct\kk)\overset{\alpha}{\rIn}([m],A_\bullet)\to([n],B_\bullet)$$into an active followed by an inert morphism in $\Gamma^+$ and using Section \ref{Gammavsdendrixinert}.

Conditions (1)-(2) are then satisfied because of the essential uniqueness of the active/inert factorisation system. This construction defines a functor $\Gamma^+\to\Omega$ taking inert (resp. active) morphisms to inert (resp. active) morphisms, and restricting to the comparison functor $\Gamma^+_{in}\to\Omega_{in}$ of Section \ref{Gammavsdendrixinert}.

\subsection{Proof of Proposition \ref{Gammaplus}}\label{dendrixappendixproof}We treat only the moment category $\Gamma$ because the proof for $\Delta$ is completely analogous. It remains to be shown that the comparison functor $\Gamma^+\to\Omega$ of Section \ref{Gammavsdendrixactive} induces an equivalence of categories $\Gamma^+\simeq\Omega_r$. This is clear with regard to objects and inert morphisms using Sections \ref{Gammavsdendrix} and \ref{Gammavsdendrixinert}. Since the comparison functor respects the active/inert factorisation, the image-morphism of any $\Gamma$-tree morphism $(\phi,f):([m],A_\bullet)\to([n],B_\bullet)$ factors through a reduced dendrix. It follows from the definition of a $\Gamma$-tree morphism that edges of same height (i.e. elements of $A_i$) are taken to edges of same height (namely elements of $B_{\phi(i)}$). Therefore, the comparison functor $\Gamma^+\to\Omega$ factors through $\Omega_r$.

Assume conversely we are given a \emph{reduced} dendrix morphism $S_{([m],A_\bullet)}\to S_{([n],B_\bullet)}$. The condition on the edge-heights determines a functor $\phi:[m]\to[n]$ in $\Delta$. The dendrix morphism itself is represented by a coherent family of \emph{reduced} subdendrices of $S_{([n],B_\bullet)}$, see Section \ref{Gammavsdendrixactive}. Using Section \ref{Gammavsdendrixinert} this family lifts to a coherent family of inert subobjects of $([n],B_\bullet)$ indexed by the vertex-corollas of $([m],A_\bullet)$. In particular, each edge $([0],\11)\rIn([m],A_\bullet)$ is taken to a well-defined edge $([0],\11)\rIn([n],B_\bullet)$ in accordance with $\phi$. This defines maps $f_i:A_i\to B_{\phi(i)}$ and hence a $\Gamma$-tree morphism $(\phi,f)$. This $\Gamma$-tree morphism corresponds to the given dendrix morphism because any subdendrix of a dendrix is completely determined by its leaves and its root, and the latter are consistently specified by $(\phi,f)$.\qed

\vspace{5ex}

\noindent{\small\sc Univ. C\^ote d'Azur, Lab. J. A. Dieudonn\'e, Parc Valrose, 06108 Nice, France.}\hspace{2em}\emph{E-mail:}
cberger$@$math.unice.fr


\begin{thebibliography}{99}
\bibitem{BD}J. C. Baez and J. Dolan -- \emph{Higher-dimensional algebra. III. n-categories and the algebra of opetopes}, Adv. Math. \textbf{135} (1998), 145--206.

\bibitem{B}C. Barwick -- \emph{From operator categories to higher operads}, Geom. Topol. \textbf{22} (2018), 1893--1959.

\bibitem{Ba}M. A. Batanin -- \emph{Monoidal globular categories as a natural environment for the theory of weak $n$-categories}, Adv. Math. \textbf{136} (1998), 39--103.

\bibitem{Ba2}M. A. Batanin -- \emph{The Eckmann-Hilton argument and higher operads}, Adv. Math. \textbf{217} (2008), 334--385.

\bibitem{BB}M. A. Batanin and C. Berger -- \emph{Homotopy theory for algebras over polynomial monads}, Theory Appl. Categ. \textbf{32} (2017), 148--253.

\bibitem{BaM}M. A. Batanin and M. Markl -- \emph{Operadic categories and duoidal Deligne's conjecture}, Adv. Math. \textbf{285} (2015), 1630--1687.

\bibitem{BaM2}M. A. Batanin and M. Markl -- \emph{Operadic categories as a natural environment for Koszul duality}, arXiv:1812.02935v3 (not the most recent v4!).

\bibitem{BaM3}M. A. Batanin and M. Markl -- \emph{Koszul duality for operadic categories}, arXiv:2105.05198v1.

\bibitem{Be0}C. Berger -- \emph{Combinatorial models for real configuration spaces and $E_n$-operads}, Contemp. Math. \textbf{202} (1997), 37--52.

\bibitem{Be}C. Berger -- \emph{A cellular nerve for higher categories}, Adv. Math. \textbf{169} (2002), 118--175.

\bibitem{Be2}C. Berger -- \emph{Iterated wreath product of the simplex category and iterated loop spaces}, Adv. Math. \textbf{213} (2007), 230--270.

\bibitem{BMW}C. Berger, P.-A. Melli\`es and M. Weber -- \emph{Monads with arities and their associated theories}, J. Pure Appl. Math. \textbf{216} (2012), 2029--2048.

\bibitem{BeMo}C. Berger and I. Moerdijk -- \emph{Resolution of coloured operads and rectification of homotopy algebras}, Contemp. Math. \textbf{431} (2007), 31--58.

\bibitem{BV}J. M. Boardman and R. M. Vogt -- \emph{Homotopy invariant algebraic structures on topological spaces}, Lect. Notes Math. \textbf{347} (1973).

\bibitem{BF}A. K. Bousfield and E. M. Friedlander -- \emph{Homotopy theory of $\Gamma$-spaces, spectra, and bisimplicial sets}, Lect.
Notes Math. \textbf{658} (1978), 80--130.

\bibitem{CM}D.-C. Cisinski and I. Moerdijk -- \emph{Dendroidal Segal spaces and $\infty$--operads}, J. Topol. \textbf{6} (2013), 675--704.

\bibitem{CHa}H. Chu and P. Hackney -- \emph{On rectification and enrichment of infinity properads},  J. Lond. Math. Soc. (2) \textbf{105} (2022), 1--100.

\bibitem{CH}H. Chu and R. Haugseng -- \emph{Homotopy-coherent algebra via Segal conditions}, Adv. Math. \textbf{385(4)} (2021), 107733.

\bibitem{CHH}H. Chu, R. Haugseng and G. Heuts -- \emph{Two  models  for  the homotopy  theory of $\infty$-operads}, J. Topol. \textbf{11} (2018), 856--872.

\bibitem{CL1}J. R. B. Cockett and S. Lack -- \emph{Restriction categories I: categories of partial maps}, Theor. Comp. Sci. \textbf{270} (2002), 223-259.

\bibitem{GKT}I. G\'alvez-Carrillo, J. Kock and A. Tonks -- \emph{Decomposition spaces, incidence algebras and M\"obius inversion I: Basic theory}, Adv. Math. \textbf{331} (2018), 952--1015.

\bibitem{G}E. Getzler -- \emph{Operads revisited}, Progr. Math. \textbf{269} (2009), 675--698.

\bibitem{GK}E. Getzler and M. Kapranov -- \emph{Modular operads}, Compositio Math. \textbf{110} (1998), 65--126.

\bibitem{GiK}V. Ginzburg and M. Kapranov -- \emph{Koszul duality for operads}, Duke Math. J. \textbf{76} (1994), 203--272.

\bibitem{H}P. Hackney -- \emph{Categories of graphs for operadic structures}, arXiv:2109.06231.

\bibitem{HRY}P. Hackney, M. Robertson and D. Yau -- \emph{Infinity Properads and Infinity Wheeled Properads}, Lect. Notes Math. \textbf{2147} (2015).

\bibitem{HRY1}P. Hackney, M. Robertson and D. Yau -- \emph{A simplicial model for infinity properads}, High. Struct. \textbf{1} (2017), 1--21.

\bibitem{HRY2}P. Hackney, M. Robertson and D. Yau -- \emph{A graphical category for higher modular operads}, Adv. Math. \textbf{365} (2020), 107044.

\bibitem{HM}G. Heuts and I. Moerdijk -- \emph{Simplicial and dendroidal homotopy theory,} Springer, Ergebnisse der Mathematik und ihrer Grenzgebiete (3) \textbf{75} (2022), xx+612 pp.

\bibitem{HHM}G. Heuts, V. Hinich and I. Moerdijk -- \emph{On the equivalence between Lurie's model and the dendroidal model for infinity-operads}, Adv. Math. \textbf{302} (2016), 869--1043.

\bibitem{J}A. Joyal -- \emph{Disks, duality and $\theta$-categories}, preprint (1997).

\bibitem{KM}R. M. Kaufmann and M. Monaco -- \emph{Plus constructions, plethysm, and unique factorization categories with applications to graphs and operad-like theories}, arXiv:2209.06121.

\bibitem{KW}R. M. Kaufmann and B. C. Ward -- \emph{Feynman categories}, Ast\'erisque \textbf{387} (2017), vii+161 pp.

\bibitem{K0}J. Kock -- \emph{Polynomial  functors  and  trees}, Int. Math. Res. Notices \textbf{2011} (2011), 609--673.

\bibitem{K}J. Kock -- \emph{Graphs, hypergraphs, and properads}, Collect. Math. \textbf{67} (2016), 155--190.

\bibitem{K2}J. Kock -- \emph{The incidence comodule bialgebra of the Baez-Dolan construction}, Adv. Math. \textbf{383} (2021), 107693.

\bibitem{Le}T. Leinster -- \emph{Higher operads, higher categories}, London Math. Soc. Lect. Note Series \textbf{298} (2004),  xiv+433 pp.

\bibitem{Lu}J. Lurie -- \emph{Higher Algebra}, http://math.harvard.edu/~lurie/papers/higheralgebra.pdf

\bibitem{MacL}S. Mac Lane -- \emph{Categorical Algebra}, Bull. Amer. Math. Soc. \textbf{71} (1965) 40--106.

\bibitem{MSS}S. Margolis, F. Saliola and B. Steinberg -- \emph{Combinatorial topology and the global dimansion of algebras arising in combinatorics}, J. Eur. Math. Soc. \textbf{17} (2015), 3037--3080.

\bibitem{May}P. May -- \emph{The geometry of iterated loop spaces}, Lect. Notes in Math. \textbf{271} (1972).

\bibitem{MW}I. Moerdijk and I. Weiss -- \emph{Dendroidal sets}, Algebr. Geom. Topol. \textbf{7} (2007), 1441--1470.

\bibitem{R}Ch. Rezk -- \emph{A Cartesian presentation of weak $n$-categories}, Geom. Topol. \textbf{14} (2010), 521--571.

\bibitem{Se}G. Segal -- \emph{Categories and cohomology theories}, Topology \textbf{13} (1974), 293--312.

\bibitem{V}B. Vallette -- \emph{A Koszul duality for PROPs}, Trans. Amer. Math. Soc. \textbf {359} (2007), 4865--4943.

\bibitem{W}M. Weber -- \emph{Familial $2$-functors and parametric right adjoints}, Theory Appl. Cat. \textbf{22} (2007), 665--732.

\end{thebibliography}
\end{document}